\documentclass{amsart}    
\usepackage{amsthm, amsmath, amscd, amssymb}
\theoremstyle{plain}
\newtheorem{theorem}{Theorem}[section]
\newtheorem{lemma}[theorem]{Lemma}
\newtheorem{proposition}[theorem]{Proposition}
\newtheorem{corollary}[theorem]{Corollary}
\newtheorem{conjecture}[theorem]{Conjecture}
\theoremstyle{definition}
\newtheorem{definition}[theorem]{Definition}

\newtheorem{example}[theorem]{Example}
\theoremstyle{remark}
\newtheorem{remark}[theorem]{Remark}
\newtheorem*{ack}{Acknowledgement}
\usepackage[mathscr]{eucal}
\usepackage{graphics, graphpap}
\usepackage{array, tabularx, longtable}

\def\ACVF{\text{ACVF}}
\def\VF{\mathrm{VF}}

\def\loc{\mathrm{loc}}
\def\sr{\mathrm{sr}}

\def\RV{\mathrm{RV}}
\def\RES{\mathrm{RES}}

\def\HL{\mathrm{HL}}
\def\hl{\mathfrak{\mathtt{H}\!\mathtt{L}}}

\def\rv{\mathrm{rv}}
\def\val{\mathrm{val}}
\def\Var{\mathrm{Var}}

\def\sp{\mathrm{sp}}

\def\Gr{\mathrm{Gr}}

\def\ord{\mathrm{ord}}
\def\Jac{\mathrm{Jac}}
\def\alg{\text{alg}}

\def\Spec{\mathrm{Spec}}

\def\bdd{\mathrm{bdd}}
\def\vol{\mathrm{vol}}

\def\Ob{\mathrm{Ob}}
\def\Rig{\text{Rig}}
\def\sp{\mathrm{sp}}

\def\fin{\text{fin}}

\def\SL{\text{SL}}

\def\Gr{\mathrm{Gr}}
\def\gcd{\mathrm{gcd}}

\def\Hom{\mathrm{Hom}}

\def\GBSRig{\mathrm{GBSRig}}
\def\BSRig{\mathrm{BSRig}}

\def\SE{\mathfrak{\mathtt{M}\!\mathtt{V}}}
\def\ka{\kappa}
\def\k{\mathbf{k}}
\def\t{\mathbf{t}}

\def\x{\mathbf{x}}

\def\L{\mathbb{L}}

\title[The integral identity conjecture]{\bf Proofs of the integral identity conjecture over algebraically closed fields}  
\author{L\^e Quy Thuong}
\dedicatory{\it Dedicated to Professor Nguy$\tilde{\text{\^e}}$n H. V. Hu\!{\scriptsize\textquoteright}ng on the occasion of his sixtieth birthday}
\address{Institut de Math\'ematiques de Jussieu, UMR 7586 CNRS, 4 place Jussieu, 75005 Paris, France {\rm(current)}}
\email{leqthuong@math.jussieu.fr}
\address{Department of Mathematics, Vietnam National University, 334 Nguyen Trai Street, Hanoi, Vietnam}
\email{thuonglq@vnu.edu.vn}

\begin{document}           
\begin{abstract}
Recently, it is well known that the conjectural integral identity is of crucial importance in the motivic Donaldson-Thomas invariants theory for non-commutative Calabi-Yau threefolds. The purpose of this article is to consider different versions of the identity, for regular functions and formal functions, and to give them the positive answer for the ground field algebraically closed. Technically, the result on motivic Milnor fiber by Hrushovski-Loeser using Hrushovski-Kazhdan's motivic integration and Nicaise's computations on motivic integrals on special formal schemes are main tools.
\end{abstract}
\maketitle                 

\section{Introduction}\label{introduction}
\subsection{}\label{firstsub}
Throughout this article, $\ka$ will be a field of characteristic zero. For $m\geq 1$, let $\mu_m$ denote $\Spec\Big(\ka[t]/(t^m-1)\Big)$, the group scheme of $m$th roots of unity, and let $\hat{\mu}$ be the limit of the projective system $(\mu_m)_{m\geq 1}$ whose transition morphisms $\mu_{mn}\to\mu_m$ are given by $s\mapsto s^n$. An (algebraic) $\ka$-variety is a separated, reduced $\ka$-scheme of finite type; if $\mathscr S$ is a $\ka$-variety then by an $\mathscr S$-variety we mean an algebraic $\ka$-variety $\mathscr X$ together with a morphism $\mathscr X\to \mathscr S$. As in \cite{DL4}, let $\Var_{\mathscr S}$ be the category of $\mathscr S$-varieties, $K(\Var_{\mathscr S})$ its Grothendieck ring, and $\mathscr M_{\mathscr S}$ the localization of $K(\Var_{\mathscr S})$ at $[\mathbb A^1]:=[\mathbb A_{\mathscr S}^1]$. Also by \cite{DL4}, a {\it good} $\mu_m$-action on an $\mathscr S$-variety $\mathscr X$ is a group action $\mu_m\times \mathscr X \to \mathscr X$ which is a morphism of $\mathscr S$-varieties each of whose orbits is contained in an affine subvariety of $\mathscr X$, and a {\it good} $\hat{\mu}$-action on $\mathscr X$ is an action of $\hat{\mu}$ on $\mathscr X$ factoring through some good $\mu_m$-action. One can also consider the category of $\mathscr S$-varieties endowed with good $\hat{\mu}$-action and its Grothendieck ring $\mathscr M_{\mathscr S}^{\hat{\mu}}$ (cf. \cite{DL4}). In the sequel, the rings $\mathscr M_{\Spec(\ka)}$ and $\mathscr M_{\Spec(\ka)}^{\hat{\mu}}$ will be rewritten simply by $\mathscr M_{\ka}$ and $\mathscr M_{\ka}^{\hat{\mu}}$, respectively. 

If $\mathscr M$ is one of the previous Grothendieck rings, or $\mathscr M$ is the Grothendieck ring $!K(\RES)[[\mathbb A_{\ka}^1]^{-1}]$ in Section \ref{prepare10}, we shall denote by $\mathscr M_{\loc}$ the localization of $\mathscr M$ with respect to the multiplicative family generated by the elements $1-[\mathbb A^1]^i$, $i\geq 1$. We shall also write $\loc$ for the localization morphism $\mathscr M\to \mathscr M_{\loc}$. Moreover, if $\mathscr N$ is the previous $\mathscr M$ or $\mathscr M_{\loc}$, we denote by $\mathscr N[[T]]_{\sr}$ the sub-$\mathscr N$-module of $\mathscr N[[T]]$ generated by 1 and by finite products of terms $[\mathbb A^1]^eT^i/(1-[\mathbb A^1]^eT^i)$ with $e\in\mathbb{Z}$, $i\in\mathbb{N}_{>0}$. An element of $\mathscr N[[T]]_{\sr}$ is called a {\it rational} function in $T$. As shown in \cite{DL1}, there is a unique $\mathscr N$-linear morphism $\lim_{T\to\infty}: \mathscr N[[T]]_{\sr}\to \mathscr N$ such that
$$\lim_{T\to\infty}\left(\frac{[\mathbb{A}^1]^eT^i}{1-[\mathbb{A}^1]^eT^i}\right)=-1.$$

\subsection{}
Starting from the year 1995 with a talk by Kontsevich at Orsay \cite{Kont}, and the strong developments by Batyrev \cite{Ba1,Ba2}, Denef-Loeser \cite{DL2,DL3}, Looijenga \cite{Loo}, Sebag \cite{Se}, Loeser-Sebag \cite{LS}, Nicaise-Sebag \cite{NS, NS2}, Nicaise \cite{Ni2}, Hrushovski-Kazhdan \cite{HK,HK2}, the geometric motivic integration has risen as a power tool in the study of algebro-geometric objects over $\ka$. The original idea is to relate to a $\ka$-variety $\mathscr X$ the arc space $\mathcal L(\mathscr X)$ on which the motivic measure takes values in a completion of $\mathscr{M}_{\ka}$.

If $\mathscr Y\to \mathscr X$ is a resolution of singularities, the induced morphism $\mathcal L(\mathscr Y)\to \mathcal L(\mathscr X)$ is a bijection outside negligible subsets. It gives rise to the fundamental formula of change of variables which expresses the motivic integral on $\mathcal L(\mathscr X)$ in terms of that on $\mathcal L(\mathscr Y)$. By this formula, Kontsevich showed in his Orsay talk \cite{Kont} that two $K$-equivalent $n$-dimensional smooth proper complex varieties (e.g., two birationally equivalent complex Calabi-Yau varieties) have the same Betti numbers (in fact the same Hodge numbers). 

Motivic zeta function, motivic nearby cycles and motivic Milnor fiber are defined by Denef-Loeser \cite{DL1} in the way using motivic integration. Let $f$ be a regular function on an algebraic $\ka$-variety $\mathscr X$ of pure dimension $d$ with the zero locus $\mathscr X_0$. Put
$$\mathscr X_m(f):=\left\{\varphi\in\mathcal{L}_m(\mathscr X)\mid f(\varphi)=t^m\mod t^{m+1}\right\},$$ 
let it be endowed with a $\mu_m$-action given by $s\cdot\varphi(t)=\varphi(st)$, thus the $\mu_m$-equivariant morphism $\varphi(t)\mapsto\varphi(0)$ defines an element $[\mathscr{X}_m(f)]$ in $\mathscr{M}_{\mathscr X_0}^{\hat{\mu}}$. Then the motivic zeta function $Z_f(T)$ of $f$ is defined as follows
$$
Z_f(T):=\sum_{m\geq 1}[\mathscr{X}_m(f)][\mathbb{A}_{\mathscr X_0}^1]^{-md}T^m\ \in \mathscr M_{\mathscr X_0}^{\hat{\mu}}[[T]]$$
It is proved in \cite{DL1} that $Z_f(T)$ is a rational function, and the limit $-\lim_{T\to \infty}Z_f(T)$ in $\mathscr M_{\mathscr X_0}^{\hat{\mu}}$, denoted by $\mathcal S_f$, is called the {\it motivic nearby cycles} of $f$. In the same way, for a closed point $\x$ of $\mathscr X_0$, we put $\mathscr X_{\x,m}(f):=\{\varphi\in \mathscr X_m(f)\mid \varphi(0)=\x\}$ and consider the motivic zeta function
$$Z_{f,\x}(T):=\sum_{m\geq 1}[\mathscr{X}_{\x,m}(f)][\mathbb{A}_{\ka}^1]^{-md}T^m\ \in \mathscr M_{\ka}^{\hat{\mu}}[[T]].$$
Again by \cite{DL1}, $Z_{f,\x}(T)$ is a rational function, and the limit $\mathcal S_{f,\x}:=-\lim_{T\to \infty}Z_{f,\x}(T)$ is the {\it motivic Milnor fiber of $f$ at $\x$}. Equivalently, we have $\mathcal S_{f,\x}=(\{\x\}\hookrightarrow \mathscr X_0)^*\mathcal S_f$. 

Explicitly, it is shown in \cite{DL1} that
\begin{equation}\label{c1:eq2}
\mathcal{S}_f=\sum_{\emptyset\not=I\subset J} (1-[\mathbb A_{\mathscr X_0}^1])^{|I|-1}[\tilde{E}_I^{\circ}],
\end{equation}
where the classes $[\tilde{E}_I^{\circ}]$ are built by Denef-Loeser \cite{DL1} from data of a resolution of singularities of $(\mathscr X,\mathscr X_0)$. Also due to \cite{DL1}, the previous formula (\ref{c1:eq2}) is independent of the choice of resolution of singularities.

Now let us go to the formal context. Let $\mathfrak X$ be a special formal $\ka[[t]]$-scheme with structural morphism $\mathfrak f$ and with reduction $\mathfrak X_0$. By \cite{Tem}, there exists a resolution of singularities of the formal scheme $(\mathfrak X, \mathfrak X_0)$, so by using the formula (\ref{c1:eq2}), Kontsevich-Soibelman in \cite{KS} defined motivic nearby cycles $\mathcal{S}_{\mathfrak f}\in \mathscr M_{\mathfrak X_0}^{\hat{\mu}}$ and motivic Milnor fiber $\mathcal{S}_{\mathfrak f, \x}\in \mathscr M_{\ka}^{\hat{\mu}}$. As explained in Subsection \ref{tetdenroi}, if $\ka$ is algebraically closed, these are independent of the uniformizing parameter $t$. By Subsection \ref{mncs}, $\mathcal{S}_{\mathfrak f}$ (hence $\mathcal{S}_{\mathfrak f,\x}$) is independent of the choice of the resolution of singularities. Remark that this approach also definitely agrees with Nicaise-Sebag's formula of the analytic Milnor fiber \cite{NS}, with Nicaise's formula of the motivic volume \cite{Ni2}, and with Jiang's definition of the motivic Milnor fiber of a cyclic $L_{\infty}$-algebra \cite{Ji}.  

\subsection{}
In \cite{KS}, Kontsevich and Soibelman introduced in an important way the concept of motivic Donaldson-Thomas invariants in the framework of non-commutative Calabi-Yau threefolds over a field of characteristic zero. Among others, the derived category of coherent sheaves on a compact/local Calabi-Yau threefolds is a central object of the investigation. The approach of Kontsevich-Soibelman \cite{KS} is to use motivic Milnor fiber instead of topological Milnor fiber as for the classical Donalson-Thomas invariants theory. In particular, To\"{e}n \cite{To} developed knowledges of the derived Hall algebra, and based on this, Kontsevich-Soibelman \cite{KS} studied the motivic Hall algebra.

The foundation of Kontsevich and Soibelman's theory includes the integral identity conjectured in \cite[Conj. 4.4]{KS}. The role of this identity is crucial, because, shortly speaking, the existence of the motivic Donaldson-Thomas invariants would be suspended if the integral identity were not proved. 

Here is the statement of the integral identity conjecture. Fix a system of coordinates $(x,y,z)$ of the $\ka$-vector space $\ka^d=\ka^{d_1}\times \ka^{d_2}\times \ka^{d_3}$. 

\begin{conjecture}[Kontsevich-Soibelman \cite{KS}, Conj. 4.4]\label{conj1}
Let $f$ be in $\ka[[x,y,z]]$ invariant by the natural $\ka^{\times}$-action of weight $(1,-1,0)$, with $f(0,0,0)=0$. Let $\mathfrak X$ be the formal completion of $\mathbb A_{\ka}^d$ along $\mathbb A_{\ka}^{d_1}$ with structural morphism $f_{\mathfrak X}$ induced by $f$, let $\mathfrak Z$ be the formal completion of $\mathbb A_{\ka}^{d_3}$ at the origin with structural morphism $f_{\mathfrak Z}$ induced by $f(0,0,z)$. Then the integral identity 
$$\int_{\mathbb A_{\ka}^{d_1}}\mathcal{S}_{f_{\mathfrak X}}=[\mathbb A_{\ka}^1]^{d_1}\mathcal{S}_{f_{\mathfrak Z}}$$ 
holds in $\mathscr M_{\ka}^{\hat{\mu}}$. Here $\int_{\mathbb A_{\ka}^{d_1}}$ denotes the forgetful morphism $\mathscr M_{\mathbb A_{\ka}^{d_1}}^{\hat{\mu}}\to \mathscr M_{\ka}^{\hat{\mu}}$.
\end{conjecture}

\subsection{}
First, let us study the regular version of Conjecture \ref{conj1}, that is for a regular function $f$. This version was already considered in \cite{Thuong}, it was verified when $f$ is either a composition of a polynomial in two variables with a pair of regular functions with no variable in common, or of Steenbrink type. More precisely, the conjecture holds in the valid scope of formulas of Guibert-Loeser-Merle \cite{GLM1,GLM2}, under certain additional conditions. Herein, we continue the work in the context of localization and under the condition that $\ka$ is algebraically closed. 

\begin{theorem}\label{conj}
Let $f$ be in $\ka[x,y,z]$ invariant by the natural $\ka^{\times}$-action of weight $(1,-1,0)$, and $f(0,0,0)=0$. Let $i$ denote the inclusion of $\mathbb{A}_{\kappa}^{d_1}$ in $\mathscr X_0$. If $\ka$ is an algebraically closed field, then the elements $\int_{\mathbb{A}_{\kappa}^{d_1}}i^*\mathcal{S}_f$ and $[\mathbb{A}_{\ka}^1]^{d_1}\mathcal{S}_{f|_{\mathbb{A}_{\ka}^{d_3}},0}$ of $\mathscr{M}_{\ka}^{\hat{\mu}}$ have the same image in $\mathscr{M}_{\ka,\loc}^{\hat{\mu}}$ by localization. 
\end{theorem} 

Our proof is strongly inspired from the computations of Hrushovski-Loeser on motivic Milnor fiber \cite{HL} using the Hrushovski-Kazhdan's motivic integration \cite{HK,HK2}. Consider the theory of algebraically closed valued fields of equal characteristic zero $\ACVF(0,0)$ with base structure $\ka((t))$ (see  \cite{HK}). According to \cite{HL}, one can construct a ``natural'' morphism of rings $\hl$ from $K(\vol\VF^{\bdd}[*])$, the Grothendieck ring of bounded $\ka((t))$-definable subsets of $\VF^n\times\RV^{\ell}$'s endowed with a volume form, to $\mathscr{M}_{\ka,\loc}^{\hat{\mu}}$. By this morphism, Theorem \ref{conj} is now equivalent to the claim that $\hl([X])=\hl([X_0])$, where
\begin{align*}
X:=\left\{(x,y,z)\in\VF^d\
\begin{array}{ |l}
\val(x)\geq 0, \val(y)>0, \val(z)>0 \\ 
\rv(f(x,y,z))=\rv(t)
\end{array}
\right\},
\end{align*}
and
$$X_0:=\left\{(x,y,z)\in X \mid x=0\ \text{or}\ y=0\right\}.$$
Here, by definition, $\val(x):=\min_i\{\val(x_i)\}$. Put 
\begin{equation}\label{hehehihi}
X_1:=X-X_0
\end{equation}
and prove that $\hl([X_1])=0$. We partition $X_1$ along the map $X_1\to\Gamma_{>0}$ given by $(x,y,z)\mapsto \val(x)+\val(y)$ into the fibers $X_{1,\gamma}$ over $\gamma\in\Gamma_{>0}$. Fix an algebraic closure $\ka((t))^{\alg}$ of $\ka((t))$. The action of the group $\ka((t))^{\alg}[\tau,\tau^{-1}]$ on $X_{1,\gamma}$ by fixing the sum $\val(x)+\val(y)$ (equal to $\gamma$), namely, $\tau(x,y,z)=(\tau x,\tau^{-1}y,z)$, is a free action. Then the natural projection from $X_{1,\gamma}$ to the quotient space $\tilde{X}_{1,\gamma}$ by this action is a fibration whose fibers are isomorphic as ``half-closed" annuli of modulus $\gamma$. The image of the class of such an annulus under $\hl$ is equal to zero. Then, using appropriate properties of motivic integration, the proof is completed (see Subsection \ref{bbb}).

\subsection{}
We also introduce a proof of the integral identity localized when the field $\ka$ is algebraically closed. The proof is a combination of the theory of geometric motivic integration and the development of the integration theory for valued fields inspired from the model theory.

\begin{theorem}\label{full}
If $\ka$ is algebraically closed, then Conjecture \ref{conj1} is true up to the localization $\loc$ (cf. Subsec. \ref{firstsub}).
\end{theorem}

\begin{remark}
If the localization morphism $\loc$ were injective, Theorem \ref{full} would imply that Conjecture \ref{conj1} is true. At present that $\loc$ is injective is also an open problem. Fortunately, it is clear that Theorem \ref{full} is sufficiently useful for the theory of motivic Donaldson-Thomas invariants \cite{KS}, where Conjecture \ref{conj1} is one of important ingredients, because the setting of the theory is among $\mathscr{M}_{\ka,\loc}^{\hat{\mu}}$, $\mathscr{M}_{\ka,\loc}^{\hat{\mu}}[[\mathbb A_{\ka}^1]^{1/2}]$ or localizations of these rings. 
\end{remark}

Geometric motivic integration was defined in the framework of separated formal schemes topologically of finite type, over the formal spectrum of a discretely value ring with perfect residue field, has been first introduced in Sebag's work \cite{Se}, and in the framework of the classical rigid analytic spaces by Loeser-Sebag \cite{LS}. Some extensions of these constructions in the context of the formal/rigid geometry has also be developed in the works of Nicaise-Sebag \cite{NS, NS1, NS2} and Nicaise \cite{Ni2}. A key tool of \cite{LS}, \cite{NS, NS2} is the concept of N\'eron models and N\'eron smoothenings in motivic integration. By using this, Nicaise is able to extend in \cite{Ni2} various results of \cite{LS}, \cite{NS, NS2} to those in the framework of special formal schemes.

Let us consider Theorem \ref{full} in terms of \cite[Cor. 7.13]{Ni2} (this result extends \cite[Cor. 7.7]{NS}) that concerns the motivic volume $S(\mathfrak X,\ka((t))^{\alg})$, and we have $S(\mathfrak X,\ka((t))^{\alg})= [\mathbb A_{\mathfrak X_0}^1]^{-(d-1)}\mathcal S_{f_{\mathfrak X}}$ (here $\mathfrak X_0=\mathbb A_{\ka}^{d_1}$). The image $\SE([\mathfrak X_{\eta}])$ of the motivic volume under the forgetful morphism only depends on $\mathfrak X_{\eta}$ and it satisfies the identity
$$\int_{\mathbb A_{\ka}^{d_1}}\mathcal S_{f_{\mathfrak X}}=[\mathbb A_{\ka}^1]^{d-1}\SE([\mathfrak X_{\eta}]).$$
According to Corollary \ref{sosanh}, $\SE$ defines a group morphism from the Grothendieck ring $K(\BSRig_{\ka((t))^{\alg}})$ of bounded smooth rigid $\ka((t))^{\alg}$-varieties to $\mathscr M_{\ka}^{\hat{\mu}}$. Decomposing $\mathfrak X_{\eta}$ into a disjoint union of $X'_0$ and $X'_1$ subject to the conditions $(x=0\ \text{or}\ y=0)$ and $(x\not=0\ \text{and}\ y\not=0)$, we are able to prove that 
$$[\mathbb A_{\ka}^1]^{d_1}\mathcal S_{f_{\mathfrak Z}}=[\mathbb A_{\ka}^1]^{d-1}\SE([X'_0])$$ 
in the ring $\mathscr M_{\ka}^{\hat{\mu}}$. Moreover, if $\ka$ is an algebraically closed field, Hrushovski-Kazhdan's integration \cite{HK, HK2} can be applied to the theory $\ACVF(0,0)$ for rigid varieties with base structure $\ka((t))$. Indeed, in Section \ref{sec66}, we show that the identity
$$\loc\left([\mathbb A_{\ka}^1]^{d-1}\SE([X'_1])\right)=\hl([X_1^{\prime\star}])$$
holds in $\mathscr M_{\ka,\loc}^{\hat{\mu}}$, where $X_1^{\prime\star}$ is nothing but $X_1$ in (\ref{hehehihi}). As above $\hl([X_1])=0$, thus Theorem \ref{full} follows. 

\setcounter{tocdepth}{1} 
\tableofcontents

\section{Preliminaries on the theory $\ACVF(0,0)$}\label{prepare1}
In this section, we shall use \cite{HK,HK2} and \cite{HL} as principal references.
\subsection{Notation}
Let us consider the theory $\ACVF(0,0)$ of algebraically closed valued fields of equal characteristic zero, which has two sorts $\VF$ and $\RV$. The language on $\ACVF(0,0)$ consists of 

\begin{itemize}
  \item the language of rings on the $\VF$-sort, 
  \item the language on $\RV$-sort with abelian group operations $\cdot$, $/$, a unary predicate $\k^{\times}$ for a subgroup, a binary operation $+$ on $\k=\k^{\times}\cup\{0\}$, and 
  \item the function notation $\rv$ for a function $\VF^{\times}\to \RV$. 
\end{itemize} 
The theory states that $\VF$ is a valued field, with valuation ring $R$ and maximal ideal $\mathfrak m$, $\rv: \VF^{\times}\to \RV$ is a surjective morphism of groups, and the restriction to $R$ (augmented by $0\mapsto 0$) is a surjective morphism of rings. For an ordered abelian group $A$, the structure $\ACVF_A(0,0)$ induces on $\Gamma$ a uniquely divisible abelian group, with constants for the elements of $\Gamma(A)$. Thus, every definable subset of $\Gamma$ is a finite union of points and open intervals.

There are identities $\RV=\VF^{\times}/(1+\mathfrak m)$, $\Gamma=\VF^{\times}/R^{\times}$ and $\k=R/\mathfrak m$, and an exact sequence of morphisms of groups
$$0\to \k^{\times}\to \RV\stackrel{\val_{\rv}}{\to}\Gamma\to 0.$$
There are natural maps $\rv: \VF\to \RV$, $\val: \VF\to \Gamma$ and $\val_{\rv}: \RV\to \Gamma$. By definition, $R$ and $\mathfrak m$ are the non-archimedian closed and open unit discs, respectively.

\subsection{Definable bounded sets}
Fix a base field, an ordered abelian group $A$, and let $n$ be a natural number. 

Let $\Gamma_A[n]$ be the category of $A$-definable subsets of $\Gamma^n$, in which a morphism is an $A$-definable bijection. Denote now by $\vol\Gamma_A[n]$ the subcategory of $\Gamma_A[n]$ such that each morphism $h:X\to Y$ satisfies the condition 
$$|x|=|h(x)|,$$ 
where $|x|:=\sum_ix_i$. In the sequels, we concern categories $\vol\Gamma_A^{2\bdd}[n]$, $\vol\Gamma_A^{\bdd}[n]$, which are the full subcategories of $\vol\Gamma_A[n]$ with objects bounded, bounded below, respectively.

We also consider the category $\RV_A[n]$, resp. $\vol\RV_A[n]$, of pairs $(X,f)$ in which $X$ is an $A$-definable subset of $\RV^k$, for some integer $k\geq 0$, $f:X\to \RV^n$ is a finite-to-one map (i.e., $\Ob\vol\RV_A[n]=\Ob\RV_A[n]$). A morphism from $(X,f)$ to $(Y,g)$ in $\RV_A[n]$, resp. in $\vol\RV_A[n]$, is a definable bijection $h: X\to Y$, resp. a definable bijection $h: X\to Y$ satisfying 
$$|\val_{\rv}f(x)|= |\val_{\rv}g(h(x))|$$ 
for every $x$ in $X$. Here, by $\val_{\rv}(y)$ for $y=(y_1,\dots,y_n)\in\RV^n$ we means the $n$-tuples $(\val_{\rv}(y_1),\dots,\val_{\rv}(y_n))$. Let $\vol\RV_A^{\bdd}[n]$, resp. $\vol\RV_A^{2\bdd}[n]$, be the full subcategory of $\vol\RV_A[n]$ whose objects have $\Gamma_A$-image contained in $[\gamma,\infty]^n$, resp. in $[\gamma,\delta]^n$, for some $\gamma, \delta\in\Gamma_A$. Let $\RES[n]$, resp. $\vol\RES[n]$, denote the full subcategory of $\RV[n]$, resp. $\vol\RV[n]$, such that $\Gamma_A$-image of its objects is a finite set.

For $\gamma$ in $\Gamma$, we set $V_{\gamma}=\{x\in \RV\mid \val_{\rv}(x)=\gamma\}\cup \{0\}$, a 1-dimensional $\k$-vector space. Let $\RES$ denote the generalized residue structure as explained in \cite[Subsec. 2.2]{HL}, it consists of the definable sets $V_{\Gamma}$ ($\gamma$ in $\Gamma$) and weighted polynomial functions. We also can consider $\RES$ as a category whose objects are definable subsets of a finite products of $V_{\gamma}$'s and whose morphisms are definable bijections.

The category $\VF_A[n]$ is by definition the category of definable subsets of $n$-dimensional varieties over $K$. In other words, an object of $\VF_A[n]$ is an $A$-definable subset $X$ of $\VF^k\times\RV^{\ell}$, for some $k, \ell\geq 0$, which admits a finite-to-one map $X\to\VF^n$. The category $\vol\VF[n]$ is defined as follows. Objects of $\vol\VF[n]$ are exactly those of $\VF_A[n]$. A morphism $(X,f)\to (Y,g)$ in $\vol\VF[n]$ is an $A$-definable bijection $h:X\to Y$ such that 
$$\val\left(\Jac_h(x)\right)=0$$
for every $x$ in $X$ outside a variety of dimension $<n$. The category $\vol\VF_A^{\bdd}[n]$ is defined to be the full subcategory of $\vol\VF_A[n]$ of objects $(X,f)$ with $f(X)$ bounded.

From now on, we shall omit the subscript $A$ whenever possible. From above categories $\mathscr C[n]$, one can define the category $\mathscr C[*]$ as the direct sum $\bigoplus_{n\geq 0}\mathscr C[n]$. If $\mathscr C$ is one of the previous categories, we shall denote by $K_+(\mathscr C)$ and $K(\mathscr C)$ its Grothendieck semi-ring and Grothendieck ring, respectively.


\section{From definable $\VF$-sets to algebraic $\ka$-varieties}\label{prepare10}
We shall consider the theory $\ACVF(0,0)$ with base structure $\ka((t))$ in Sections \ref{prepare10} and \ref{proofthm1}; naturally, we shall define $\val(t)=1$.

\subsection{}
By \cite{HK2}, there is a canonical morphism of Grothendieck semi-rings 
\begin{align}\label{eq7}
\Psi: K_+(\vol\RES[*])\otimes K_+(\vol\Gamma^{\bdd}[*])\to K_+(\vol\RV^{\bdd}[*]),
\end{align}
where $\ker(\Psi)$ is generated by $[\val_{\rv}^{-1}(\gamma)]_1\otimes 1 - 1\otimes [\gamma]_1$, with $\gamma$ definable in $\Gamma$. The subscript 1 means that the classes are in degree 1. By \cite{HL}, the restriction of the morphism (\ref{eq7}) to $K_+(\vol\RES[*])\otimes K_+(\vol\Gamma^{2\bdd}[*])$ yields a canonical morphism
\begin{align}\label{eq8}
\Psi: K_+(\vol\RES[*])\otimes K_+(\vol\Gamma^{2\bdd}[*])\to K_+(\vol\RV^{2\bdd}[*]).
\end{align}

Now for any integer $n\geq 0$ one defines a map $\L: \Ob\vol\RV[n]\to \Ob\vol\VF[n]$ sending a pair $(X,f)$ to the set $\{(y_1,\dots,y_n,x)\in \VF^n\times X \mid \rv(y_i)=f_i(x)\}$. It induces by Lemma 3.21 of \cite{HK2} a canonical morphism of semi-rings
\begin{align}\label{cachan}
\int: K_+(\vol\VF^{\bdd}[n])\to K_+(\vol\RV^{\bdd}[n])/I'_{\sp},
\end{align}
where $I'_{\sp}$ is the congruence generated by $[1]_1=[\RV^{>0}]_1$ with the constant volume form $0$ in $\Gamma$. It satisfies the property that $[X]=[\L(V)]$ in $K_+(\vol\VF^{\bdd}[n])$ if and only if $\int([X])=[V]+ I'_{\sp}$ in $K_+(\vol\RV^{\bdd}[n])/I'_{\sp}$. This morphism induces a morphism between corresponding rings, which we shall also denote by $\int$.


\subsection{The morphisms $h_m$ and $\tilde{h}_m$}\label{hm2012}
Intuitively, each element of $\RV$ has the form $\alpha t^a$ with $\alpha\in\k^{\times}$ and $a\in\Gamma$. Furthermore, if $X$ is a definable subset of $\RES^n$, its elements are $n$-tuples $(\alpha_1 t^{a_1},\dots, \alpha_n t^{a_n})$ where $\alpha_i\in \k^{\times}$ and all $a_i$ belong to a finite set $F$. Identifying all the elements of $F$, $X$ is nothing but a definable subset of the affine variety $\mathbb A_{\ka}^n$. This motivates the definition of $!K(\vol\RES[*])$ and $!K(\RES)$. Precisely, $!K(\vol\RES[*])$ (resp. $!K(\RES)$) is the quotient of $K(\vol\RES[*])$ (resp. $K(\RES)$) subject to $[\val_{\rv}^{-1}(a)]=[\val_{\rv}^{-1}(0)]$ for $a$ running over $\Gamma$.

Let us recall the construction of $h_m$ and $\tilde{h}_m$ in the article \cite[Subsec. 8.2]{HL}. For an integer $m\geq 1$ and for a bounded definable subset $\Delta$ of $\Gamma^n$ (recall that ``bounded'' means ``two-sided bounded''), putting
\begin{align}\label{nice1}
a_m(\Delta)=\sum_{\gamma\in\Delta\cap(1/m \mathbb Z)^n}[\mathbb A_{\ka}^1]^{-m|\gamma|}([\mathbb A_{\ka}^1]-1)^n ,
\end{align}
one defines a morphism of rings 
$$a_m: K(\vol\Gamma^{2\bdd}[*])\to !K(\RES)[[\mathbb A_{\ka}^1]^{-1}].$$ 
It is shown in \cite[Subsec. 8.2]{HL} that, if $\Delta$ is a bounded below definable subset of $\Gamma^n$, the RHS of (\ref{nice1}), now an infinite sum, is an element of $!K(\RES)[[\mathbb A_{\ka}^1]^{-1}]_{\loc}$. This induces a morphism of rings 
$$\tilde{a}_m: K(\vol\Gamma^{\bdd}[*])\to !K(\RES)[[\mathbb A_{\ka}^1]^{-1}]_{\loc}.$$

Now, let $X=(X,f)$ be in $\RES[n]$ with $f(X)\subset V_{\gamma_1}\times\cdots\times V_{\gamma_n}$, $\gamma=(\gamma_1,\dots,\gamma_n)$. We denote by $[1]_1$ the class of $\{1\}$ in $K(\vol\RES[1])$. Setting 
\begin{equation*}
b_m(X)=
\begin{cases}
[X]\left(\frac{[1]_1}{[\mathbb A_{\ka}^1]}\right)^{m|\gamma|}\ &\text{if}\  m\gamma\in \mathbb Z^n\\
0\ &\text{otherwise}
\end{cases}
\end{equation*}
one gives rise to a morphism of rings $b_m: K(\vol\RES[*])\to !K(\vol\RES[*])[[\mathbb A_{\ka}^1]^{-1}]$. Then, composing this morphism with the canonical forgetful morphism 
$$!K(\vol\RES[*])[[\mathbb A_{\ka}^1]^{-1}]\to !K(\RES)[[\mathbb A_{\ka}^1]^{-1}]$$ 
yields the following morphism of rings, also denoted by $b_m$,
$$b_m: K(\vol\RES[*])\to !K(\RES)[[\mathbb A_{\ka}^1]^{-1}].$$
We also consider the morphism 
$$\tilde{b}_m: K(\vol\RES[*])\to !K(\RES)[[\mathbb A_{\ka}^1]^{-1}]_{\loc}$$ 
which is the composition of $b_m$ with the localization morphism.

The tensor product of $b_m$ and $a_m$ is a morphism 
$$b_m\otimes a_m: K(\vol\RES[*])\otimes K(\vol\Gamma^{2\bdd}[*])\to !K(\RES)[[\mathbb A_{\ka}^1]^{-1}]$$
that satisfies the condition that $\ker(\Psi)\subset \ker(b_m\otimes a_m)$. Indeed, let us take a generator $u$ of $\ker(\Psi)$, say, $u=[\val_{\rv}^{-1}(\gamma)]_1\otimes 1 - 1\otimes [\gamma]_1$ with $\gamma\in\Gamma$. Assuming $\gamma=i/m$, we have $a_m([\gamma]_1)=[\mathbb A_{\ka}^1]^{-i}([\mathbb A_{\ka}^1]-1)$ and $[\val_{\rv}^{-1}(\gamma)]_1=[\mathbb A_{\ka}^1]-[1]_1$ in $!K_+(\vol\RES[1])$. Thus $a_m([\gamma]_1)=b_m([\val_{\rv}^{-1}(\gamma)]_1)$, it implies that $(b_m\otimes a_m)(u)=0$. Then, one deduces that $b_m\otimes a_m$ factors into $\Psi$ followed by a morphism 
$$h_m: K(\vol\RV^{2\bdd}[*])\to !K(\RES)[[\mathbb A_{\ka}^1]^{-1}],$$
that is, $b_m\otimes a_m=h_m\circ \Psi$. In the same way, the tensor product of morphisms
$$\tilde b_m\otimes \tilde a_m: K(\vol\RES[*])\otimes K(\vol\Gamma^{\bdd}[*])\to !K(\RES)[[\mathbb A_{\ka}^1]^{-1}]_{\loc}$$
induces a morphism 
$$\tilde h_m: K(\vol\RV^{\bdd}[*])\to !K(\RES)[[\mathbb A_{\ka}^1]^{-1}]_{\loc}$$
and the following diagram 
\begin{align*}
\begin{CD}
K(\vol\RV^{2\bdd}[*])@>h_m>> !K(\RES)[[\mathbb A_{\ka}^1]^{-1}]\\
@VVV @VVV\\
K(\vol\RV^{\bdd}[*]) @>\tilde h_m>>!K(\RES)[[\mathbb A_{\ka}^1]^{-1}]_{\loc}
\end{CD}
\end{align*}
is commutative. According to \cite[Lem. 8.2.2]{HL}, for every $m\geq 1$, the morphism $\tilde h_m$ vanishes on the congruence $I'_{\sp}$, since $\tilde h_m([\RV^{>0}]_1)=\tilde h_m([1]_1)=1$; thus it factors through a morphism 
$$K(\vol\RV^{\bdd}[*])/I'_{\sp}\to !K(\RES)[[\mathbb A_{\ka}^1]^{-1}]_{\loc}.$$
We denote the latter also by $\tilde h_m$. In particular, for two elements $Y$ and $Y'$ in $K(\vol\RV^{2\bdd}[*])$ having the same image in $K(\vol\RV^{\bdd}[*])/I'_{\sp}$, we have that $h_m(Y)$ and $h_m(Y')$ have the same image in $!K(\RES)[[\mathbb A_{\ka}^1]^{-1}]_{\loc}$ (cf. \cite{HL}).


\subsection{The morphism $\Upsilon$}\label{ups2012}
We shall recall the morphism $\Upsilon$ which was already constructed in \cite[Subsec. 8.5]{HL}. Let $\chi$ be the o-minimal Euler characteristic defined in \cite[Lem. 9.5]{HK}, and we consider the following morphisms  
$$\alpha: K(\vol\Gamma[*])\to!K(\RES)[[\mathbb A_{\ka}^1]^{-1}]$$
and
$$\beta: K(\vol\RES[*])\to!K(\RES)[[\mathbb A_{\ka}^1]^{-1}]$$
which are given, respectively, by $\alpha([\Delta])=\chi(\Delta)([\mathbb A_{\ka}^1]-1)^n$ for $[\Delta]\in K(\vol\Gamma[n])$ and $\beta([X])=[X]$ for $[X]$ in $K(\vol\RES[n])$. Tensoring of $\beta$ with $\alpha$ yields a morphism
$$\beta\otimes\alpha: K(\vol\RES[*])\otimes K(\vol\Gamma[*]) \to!K(\RES)[[\mathbb A_{\ka}^1]^{-1}].$$
Since $\ker(\beta\otimes \alpha) \subset \ker(\Psi)$ (cf. \cite[Subsec. 8.2]{HL}), this morphism $\beta\otimes\alpha$ induces a morphism of rings
$$\Upsilon: K(\vol\RV[*])\to!K(\RES)[[\mathbb A_{\ka}^1]^{-1}].$$
Similarly, one may define a morphism of rings 
$$K(\vol\RV^{2\bdd}[*])\to!K(\RES)[[\mathbb A_{\ka}^1]^{-1}],$$
which is also denoted by $\Upsilon$.

Proposition 8.5.1 of \cite{HL} points out the relation between $h_m$'s and $\Upsilon$ as follows. For any $Y$ in $K(\vol\RV^{2\bdd}[*])$, the formal series $\sum_{m\geq 1}h_m(Y)T^m$ lives in $!K(\RES)[[\mathbb A_{\ka}^1]^{-1}][[T]]_{\sr}$, i.e., it is a rational function, and one has 
\begin{align}\label{ooo}
\lim_{T\to\infty}\sum_{m\geq 1}h_m(Y)T^m=-\Upsilon(Y).
\end{align}

\subsection{The morphism $\tilde{\Theta}$}
Following \cite{HK} and \cite{HL}, consider a sequence $(t_m)_{m\geq 1}$ in a fixed algebraic closure $\ka((t))^{\alg}$ of $\ka((t))$ given by $t_1=t$, $t_{nm}^m=t_n$ for $n\geq 1$, and set $t_{k/m}:=t_m^k$, $\t_{k/m}:=\rv(t_{k/m})$. Let $X$ be a $\ka((t))$-definable set over $\RES$, that is, $X$ is a $\ka((t))$-definable subset of $\RV^n$ whose image under $\val_{\rv}$ is finite. Since the sorts of $\RES$ are the $\k$-vector spaces 
$$V_{k/m}:=\left\{x\in\RV \mid \val_{\rv}(x)=k/m\right\}\cup \{0\},$$
one can view $X$ as a definable subset of $\prod_{i=1}^nV_{k_i/m}$ for some $n$, $m$ and $k_i$'s. The group $\hat{\mu}$ acts on $X$ via $\mu_m$. The image $Y$ of $X$ under the $\ka((t^{1/m}))$-definable function 
$$(x_1,\dots,x_n)\mapsto (x_1/\t_{k_1/m},\dots, x_n/\t_{k_n/m})$$
is a $\ka$-definable subset of $\k^n$, that is, a constructible subset of $\mathbb A_{\ka}^n$, and it is endowed with a $\mu_m$-action induced from the one on $X$. In other words, the correspondence $X\mapsto Y$ defines a morphism of Grothendieck semi-rings $K_+(\RES)\to K_+(\Var_{\ka}^{\hat{\mu}})$, which by \cite[Lem. 10.7]{HK} and \cite[Prop. 4.3.1]{HL} induces an isomorphism 
$$\Theta:\ !K(\RES)[[\mathbb A_{\ka}^1]^{-1}]\to !K(\Var_{\ka}^{\hat{\mu}})[[\mathbb A_{\ka}^1]^{-1}].$$
Here $!K(\Var_{\ka}^{\hat{\mu}})$ stands for the quotient of $K(\Var_{\ka}^{\hat{\mu}})$ by identifying all the classes $[\mathbb G_m, \sigma]$ with $\sigma$ a $\hat{\mu}$-action on $\mathbb G_m$ induced by multiplication by roots of $1$. Composing the morphism $\Theta$ with the natural morphism $!K(\Var_{\ka}^{\hat{\mu}})[[\mathbb A_{\ka}^1]^{-1}]\to \mathscr M_{\ka}^{\hat{\mu}}$ yields a morphism of rings
$$\tilde{\Theta}:\ !K(\RES)[[\mathbb A_{\ka}^1]^{-1}] \to \mathscr M_{\ka}^{\hat{\mu}}.$$
We also denote by $\tilde{\Theta}$ the morphism at the level of localization induced by the previous morphism 
$$\tilde{\Theta}:\ !K(\RES)[[\mathbb A_{\ka}^1]^{-1}]_{\loc} \to \mathscr M_{\ka,\loc}^{\hat{\mu}}.$$


\subsection{The morphism {\rm $\hl$}}\label{submain}
Let us consider now the compositions $\hl=\tilde{\Theta}\circ \Upsilon\circ \int$ and $\hl_m=\tilde{\Theta}\circ \tilde h_m\circ \int$ that are viewed as morphisms of rings 
$$K(\vol\VF^{\bdd}[*])\to \mathscr M_{\ka,\loc}^{\hat{\mu}}.$$

Recall from \cite[Subsec. 3]{HL} that, given $\beta=(\beta_1,\dots,\beta_n)\in\Gamma^n$, a definable subset $X\subset\VF^n\times\RV^m$ is $\beta$-invariant if, for any $(x,x')\in\VF^n\times\RV^m$ and any $(y,y')\in\VF^n\times\RV^m$ with $\val(y_i)\geq\beta_i$ for $i=1,\dots,n$, both $(x,x')$ and $(x,x')+(y,y')$ simultaneously belong to either $X$ or the complement of $X$ in $\VF^n\times\RV^m$. By \cite[Lem. 3.1.1]{HL}, for any definable subset $X\subset\VF^n$ which is bounded and closed in the valuation topology, there exists a $\beta\in\Gamma^n$ such that $X$ is $\beta$-invariant. If $\beta_i=\tilde{\beta}$ for every $i=1,\dots,n$, we shall say $\tilde{\beta}$-invariant to mean $\beta$-invariant for such $X$. The following definition is also necessary: a subset of $\RV^{\ell}$ is {\it boundedly imaginary} if its image in $\Gamma^n$ under the map $\val_{\rv}$ is bounded.

We note that, because the field $\ka$ is algebraically closed, every subfield $\ka((t^{1/m}))$ of $\ka((t))^{\alg}$ is independent of the choice of $t^{1/m}$ (see Subsection \ref{tetdenroi} for proof). The group $\mu_m$ acts naturally on $\ka((t^{1/m}))$. Let $\beta\in\Gamma^n$, and let $X$ be a $\beta$-invariant $\ka((t))$-definable subset of $\VF^n\times\RV^{\ell}$ such that the projection $X\to \VF^n$ is finite-to-one. The set $X$ is also assumed to be contained in $\VF^n\times W$ with $W$ a boundedly imaginary definable subset of $\RV^{\ell}$, and that $X_w:=\{x\in\VF^n \mid (x,w)\in X\}$ is a bounded subset of $\VF^n$ for every $w\in W$. By \cite{HK}, the $\ka((t^{1/m}))$-points of $X$ are the pullback of some definable subset 
$$X[m;\beta]\subset \prod_{i=1}^n\ka[t^{1/m}]/(t^{\beta_i})\times\RV^{\ell},$$ 
and the projection $X[m;\beta]\to\VF^n$ is finite-to-one. If every component of $\beta$ is equal to $\tilde{\beta}\in\Gamma$, we shall write $X[m;\tilde{\beta}]$ instead of $X[m;\beta]$. Now, for $\beta'=(\beta'_1,\dots,\beta'_n)\in\Gamma^n$ with $\beta_i\leq\beta'_i$ for every $i=1,\dots, n$, one has the following identity
$$[X[m;\beta']]=[X[m;\beta]][\mathbb A_{\ka}^{m(|\beta'|-|\beta|)}],$$ 
which lives in $!K(\RES)$. Thus, the quantity
$$\widetilde{X}[m]:=[X[m;\beta]][\mathbb A_{\ka}^1]^{-m|\beta|+n}$$
in $!K(\RES)[[\mathbb A_{\ka}^1]^{-1}]$ is independent of the choice of $\beta$ large enough. By abuse of notation, we also write $\widetilde{X}[m]$ for its image $\tilde{\Theta}(\widetilde{X}[m])\in\mathscr M_{\ka}^{\hat{\mu}}$ under $\tilde{\Theta}$.

\begin{proposition}\label{eq010}
Given $n, m, \ell\in\mathbb N$ and $\beta\in\Gamma^n$. Let $X$ be a $\beta$-invariant $\ka((t))$-definable subset of $\VF^n\times\RV^{\ell}$ such that $X$ is contained in $\VF^n\times W$ with $W$ a boundedly imaginary definable subset of $\RV^{\ell}$, and that $X_w$ is bounded for every $w\in W$. Assume that the projection $X\to\VF^n$ is finite-to-one. Then, we have
\begin{itemize}
\item[(i)] the identity $\hl_m([X])=\loc\left(\widetilde{X}[m]\right)$ holds in $\mathscr M_{\ka,\loc}^{\hat{\mu}}$;
\item[(ii)] the formal series $\sum_{m\geq 1}\hl_m([X])T^m$ in $\mathscr M_{\ka,\loc}^{\hat{\mu}}[[T]]$ is a rational function, whose image through $\lim_{T\to\infty}$ is equal to $-\hl([X])$ in $\mathscr M_{\ka,\loc}^{\hat{\mu}}$.
\end{itemize}
\end{proposition}

\begin{proof}
(i) follows from \cite[Prop. 8.2.3]{HL}, (ii) follows from (\ref{ooo}).
\end{proof}

Remark 8.2.2 of \cite{HL} provides a simple example of $X$ satisfying this proposition. Namely, $X$ is a bounded $\beta$-invariant $\ka((t))$-definable subset of $\mathscr X(R)$, where $\mathscr X$ is a smooth $\ka$-variety which carries a volume form. 

\begin{example}\label{exam3.2}
Let $\mathscr X$ be a smooth connected affine $\ka$-variety of pure dimension $d$, and let $f$ be a non-constant regular function on $\mathscr X$ with zero locus $\mathscr X_0$. Denote by $\pi$ the reduction map $\mathscr X(R)\to \mathscr X(\k)$. For a closed point $\x$ in $\mathscr X_0$, one puts
$$\mathscr X_{f,\x}:=\left\{x\in \VF^d \mid \val(x)\geq 0, \rv(f(x))=\rv(t), \pi(x)=\x\right\}.$$
Then $\mathscr X_{f,\x}$ is a bounded $\beta$-invariant definable subset of $\VF^d$ for any $\beta\in\mathbb N_{>0}$. By \cite[Prop. 8.3.1, Cor. 8.5.3]{HL}, one has
\begin{align*}
\hl_m([\mathscr X_{f,\x}])&=\loc\left([\mathscr X_{\x,m}(f)][\mathbb A_{\ka}^1]^{-md}\right),\\
\hl([\mathscr X_{f,\x}])&=\loc\left(\mathcal S_{f,\x}\right)
\end{align*} 
in the ring $\mathscr M_{\ka,\loc}^{\hat{\mu}}$.
\end{example}

\section{Proof of the regular version (Theorem \ref{conj})}\label{proofthm1}
\subsection{Data from the polynomial $f$}\label{data111}
First of all, let $h:=f|_{\mathbb A_{\ka}^{d_3}}$ and 
\begin{equation}\label{eq11}
\begin{aligned}
X:=\left\{(x,y,z)\in\VF^d\ 
\begin{array}{|l}
\val(x)\geq 0, \val(y)>0, \val(z)>0\\
\rv(f(x,y,z))=\rv(t)
\end{array}
\right\}.
\end{aligned}
\end{equation}
This set has the properties of the $X$ in Proposition \ref{eq010} (with $\ell=0$) as explained in \cite[Rmk. 8.2.2(2)]{HL}, it also has the $\beta$-invariance by \cite[Cor. 4.2.2]{HL} and the boundedness by definition. Let us now write $X$ as a disjoint union $X=X_0 \sqcup X_1$ of its definable subsets, where
\begin{align*}
X_0&=\left\{(x,y,z)\in X \mid x=0\ \text{or}\ y=0\right\},\\
X_1&=\left\{(x,y,z)\in X \mid x\not=0\ \text{and}\ y\not=0\right\}.
\end{align*}
By the same argument, we deduce that $X_0$, $X_1$ also satisfy Proposition \ref{eq010}. In the sequel, the following equalities will be shown to be true in $\mathscr M_{\ka,\loc}^{\hat{\mu}}$:
\begin{itemize}
  \item[(i)] $\hl([X])=\loc\left(\int_{\mathbb A_{\ka}^{d_1}}i^*\mathcal S_f\right)$ (see Subsection \ref{sub42}),
  \item[(ii)] $\hl([X_0])=\loc\left([\mathbb A_{\ka}^1]^{d_1}\mathcal S_{h,0}\right)$ (see Subsection \ref{sub43}),
  \item[(iii)] $\hl([X_1])=0$ (see Subsection \ref{bbb});
\end{itemize}
and Theorem \ref{conj} is then proved.

\subsection{Computation of {\rm $\hl([X])$}}\label{sub42}
It suffices to consider $\beta=2$, that is, the definable set $X$ is viewed as to be 2-invariant; thus going back to \cite[Subsec. 4.2]{HL} we have the following. For $m\geq 1$, 
\begin{align*}
X[m;2]&=\left\{\varphi\in \left(\ka[t^{1/m}]/(t^2)\right)^d \mid \varphi(0)\in\mathbb A_{\ka}^{d_1}, \rv f(\varphi)=\rv(t)\right\}\\
&=\left\{\varphi\in \left(\ka[t^{1/m}]/(t^2)\right)^d \mid \varphi(0)\in\mathbb A_{\ka}^{d_1}, f(\varphi)=t\mod t^{(m+1)/m}\right\},
\end{align*}
which is isomorphic as a $\ka$-variety via the morphism $t^{1/m}\mapsto t$ to the $\ka$-variety
\begin{gather*}
\left\{\varphi\in \left(\ka[t]/(t^{2m})\right)^d \mid \varphi(0)\in\mathbb A_{\ka}^{d_1}, f(\varphi)=t^m\mod t^{m+1}\right\}\\
\cong \left(\mathscr X_m(f)\times_{\mathscr X_0}\mathbb A_{\ka}^{d_1}\right)\times \mathbb A_{\ka}^{(m-1)d}.
\end{gather*}
Thus
\begin{align*}
\widetilde{X}[m]&=[X[m;2]][\mathbb A_{\ka}^1]^{-2md+d}\\
&=[\mathscr X_m(f)\times_{\mathscr X_0}\mathbb A_{\ka}^{d_1}][\mathbb A_{\ka}^1]^{-md}\\
&=\int_{\mathbb A_{\ka}^{d_1}}i^*\left([\mathscr X_m(f)][\mathbb A_{\mathscr X_0}^1]^{-md}\right),
\end{align*}
and
$$\hl_m([X])=\loc\left(\int_{\mathbb A_{\ka}^{d_1}}i^*\left([\mathscr X_m(f)][\mathbb A_{\mathscr X_0}^1]^{-md}\right)\right).$$
Since 
$$\mathcal S_f=-\lim_{T\to\infty}\sum_{m\geq 1}[\mathscr X_m(f)][\mathbb A_{\mathscr X_0}^1]^{-md}T^m,$$
we deduce from Proposition \ref{eq010} that
\begin{align*}
\hl([X])&=-\lim_{T\to\infty}\sum_{m\geq 1}\hl_m([X])T^m =\loc\left(\int_{\mathbb A_{\ka}^{d_1}}i^*\mathcal S_f\right).
\end{align*}


\subsection{Computation of {\rm $\hl([X_0])$}}\label{sub43}
Similarly as previous, for $m\geq 1$, we have 
\begin{align*}
X_0[m;2]\cong 
\left\{(\varphi_1, \varphi_2,\varphi_3)\
\begin{array}{|l}
\varphi_i\in \left(\ka[t]/(t^{2m})\right)^{d_i}, i=1, 2, 3\\
\varphi_1\equiv 0\ \text{or}\ \varphi_2\equiv 0\\
\varphi_1(0)\in \mathbb A_{\ka}^{d_1}, \varphi_2(0)=0, \varphi_3(0)=0 \\ 
f(\varphi_1, \varphi_2, \varphi_3)=t^m\mod t^{m+1}
\end{array}
\right\}.
\end{align*}
Also, due to the homogeneity of $f$, namely, $f(\varphi_1, \varphi_2, \varphi_3)=f(0,0, \varphi_3)=h(\varphi_3)$ whenever $\varphi_1\equiv 0$ or $\varphi_2\equiv 0$, $X_0[m;2]$ can be written, up to isomorphism, as a cartesian product
$$Y_m\times \left(\mathscr X_{0,m}(h)\times \mathbb A_{\ka}^{(m-1)d_3}\right),$$
where 
\begin{align*}
Y_m&= 
\left\{(\varphi_1, \varphi_2)\
\begin{array}{|l}
\varphi_i\in \left(\ka[t]/(t^{2m})\right)^{d_i}, i=1, 2\\
\varphi_1\equiv 0\ \text{or}\ \varphi_2\equiv 0\\
\varphi_1(0)\in \mathbb A_{\ka}^{d_1}, \varphi_2(0)=0
\end{array}
\right\}\\
&=\left((t)\ka[t]/(t^{2m})\right)^{d_2}\sqcup \left(\left(\ka[t]/(t^{2m})\right)^{d_1}\setminus\{0\}\right)\\
&\cong \mathbb A_{\ka}^{(2m-1)d_2} \sqcup \left(\mathbb A_{\ka}^{2md_1}\setminus\{0\}\right).
\end{align*}
Thus 
\begin{align*}
\widetilde{X}_0[m]&=[X_0[m;2]][\mathbb A_{\ka}^1]^{-2md+d}\\
&=\left([\mathbb A_{\ka}^1]^{(2m-1)d_2}+[\mathbb A_{\ka}^1]^{2md_1}-1\right)[\mathbb A_{\ka}^1]^{-2md+d+(m-1)d_3}[\mathscr X_{0,m}(h)],
\end{align*}
and
\begin{align*}
\sum_{m\geq 1}\widetilde{X}_0[m]T^m&=[\mathbb A_{\ka}^1]^{d_1}\sum_{m\geq 1}[\mathscr X_{0,m}(h)][\mathbb A_{\ka}^1]^{-m(2d_1+d_3)}T^m\\
&+[\mathbb A_{\ka}^1]^{d_1+d_2}\sum_{m\geq 1}([\mathbb A_{\ka}^1]^{2md_1}-1)[\mathscr X_{0,m}(h)][\mathbb A_{\ka}^1]^{-(2d_1+2d_2+d_3)m}T^m.
\end{align*}

Now, we use properties of Hadamard product (see \cite[Subsec. 5.1]{DL3}, \cite[Lem. 7.6]{Loo} or \cite[Subsec. 8.4]{HL}). Since 
$$\sum_{m\geq 1}[\mathscr X_{0,m}(h)][\mathbb A_{\ka}^1]^{-m(2d_1+d_3)}T^m$$ 
is the Hadamard product of two series 
$$\sum_{m\geq 1}[\mathbb A_{\ka}^1]^{-2md_1}T^m=\frac{[\mathbb A_{\ka}^1]^{-2d_1}T}{1-[\mathbb A_{\ka}^1]^{-2d_1}T}$$
and 
$$\sum_{m\geq 1}[\mathscr X_{0,m}(h)][\mathbb A_{\ka}^1]^{-md_3}T^m,$$ 
it follows that 
$$
\lim_{T\to\infty}\sum_{m\geq 1}[\mathscr X_{0,m}(h)][\mathbb A_{\ka}^1]^{-m(2d_1+d_3)}T^m$$
is equal to
$$-\left(\lim_{T\to\infty}\frac{[\mathbb A_{\ka}^1]^{-2d_1}T}{1-[\mathbb A_{\ka}^1]^{-2d_1}T}\right)\cdot \left(\lim_{T\to\infty}\sum_{m\geq 1}[\mathscr X_{0,m}(h)][\mathbb A_{\ka}^1]^{-md_3}T^m\right)=-\mathcal S_{h,0}.$$
Similarly, 
$$\sum_{m\geq 1}([\mathbb A_{\ka}^1]^{2md_1}-1)[\mathscr X_{0,m}(h)][\mathbb A_{\ka}^1]^{-(2d_1+2d_2+d_3)m}T^m$$
is the Hadamard product of 
$$\sum_{m\geq 1}([\mathbb A_{\ka}^1]^{2md_1}-1)T^m=\frac{[\mathbb A_{\ka}^1]^{d_2}T}{1-[\mathbb A_{\ka}^1]^{d_2}T}-\frac{T}{1-T}$$
and
$$\sum_{m\geq 1}[\mathscr X_{0,m}(h)][\mathbb A_{\ka}^1]^{-(2d_1+2d_2+d_3)m}T^m;$$
the former has image zero under $\lim_{T\to\infty}$, thus 
$$\lim_{T\to\infty}\sum_{m\geq 1}([\mathbb A_{\ka}^1]^{2md_1}-1)[\mathscr X_{0,m}(h)][\mathbb A_{\ka}^1]^{-(2d_1+2d_2+d_3)m}T^m=0.$$
The previous arguments and Proposition \ref{eq010} then deduce that
\begin{align*} 
\hl([X_0])&=-\lim_{T\to\infty}\sum_{m\geq 1}\hl_m([X_0])T^m\\
&=-\lim_{T\to\infty}\sum_{m\geq 1}\loc\left(\widetilde{X}_0[m]\right)T^m\\
&=\loc\left([\mathbb A_{\ka}^1]^{d_1}\mathcal S_{h,0}\right).
\end{align*}


\subsection{Computation of {\rm $\hl([X_1])$}}\label{bbb}
We recall from Subsection \ref{data111} that, by definition, $X_1$ is the following definable set
\begin{align*}
X_1=\left\{(x,y,z)\in\VF^d\ 
\begin{array}{|l}
x\not=0 \ \text{and}\ y\not=0\\
\val(x)\geq 0, \val(y)>0, \val(z)>0\\
\rv(f(x,y,z))=\rv(t)
\end{array}
\right\}.
\end{align*}
Clearly, the action of the multiplicative group $G:=\mathbb G_{m,\ka((t))^{\alg}}$ on the set 
$$A:=(\VF^{d_1}-\{0\})\times (\VF^{d_2}-\{0\})\times \VF^{d_3}$$ 
given by 
$$\tau\cdot (x,y,z)=(\tau x,\tau^{-1}y,z),\ \tau\in G,\ (x,y,z)\in A,$$ 
is a free action. Then the canonical projection $A\to  A/G$ induces a surjective map $\rho: X_1\to \tilde{X}_1$, where $\tilde{X}_1$ is the image of $X_1$ under $A\to  A/G$. Observe that an element of $\tilde{X}_1$ is an orbit 
$$\tilde{\xi}_{x,y,z}:=\Big(G\cdot(x,y,z)\Big)\cap X_1=\{(\tau x,\tau^{-1}y,z)\mid -\val(x)\leq \val(\tau) <\val(y)\},$$
that is, an annulus definably isomorphic to $\mathscr B(0,r)-\mathscr B(0,r')$ for some $r,r'\in\Gamma$, where $\mathscr B(0,r)$ denotes the non-archimedean closed ball centered at $0$ of valuative radius $r$. 

\begin{lemma}
$\tilde{X}_1$ is an object in the category {\rm $\vol\VF^{\bdd}[d_3]$}.
\end{lemma} 

\begin{proof}
Let us consider the action of the multiplicative group $1+\mathfrak m$ on $A$ given by $\tau\cdot (x,y,z)=(\tau x,\tau^{-1}y,z)$. Let $\tilde{X}'_1$ be the set constructed in the same way as $\tilde{X}_1$ but with the $(1+\mathfrak m)$-action, i.e., an element of $\tilde{X}'_1$ is an orbit of the form
$$\tilde{\xi}'_{x,y,z}=\left\{(\tau x,\tau^{-1}y,z)\mid \val(\tau)=0\right\}.$$
Then the natural inclusions $\tilde{\xi}'_{x,y,z}\subset \tilde{\xi}_{x,y,z}$ induce a natural bijection between $\tilde{X}'_1$ and $\tilde{X}_1$. These two sets have the same properties, because they are made in the same method and $\tilde{\xi}'_{x,y,z}\mapsto \tilde{\xi}_{x,y,z}$ iff $\tilde{\xi}'_{x,y,z}\subset \tilde{\xi}_{x,y,z}$. So, it is enough to prove that $\tilde{X}'_1$ is in $\Ob\vol\VF^{\bdd}[d_3]$. By construction, one has 
$$A/(1+\mathfrak m)=(\VF^{d_1}-\{0\})/(1+\mathfrak m) \times (\VF^{d_2}-\{0\})/(1+\mathfrak m) \times \VF^{d_3}.$$
Note that, for $i\in\{1, 2\}$, $\VF^{d_i}-\{0\}$ is the disjoint union $\bigsqcup_I(\VF^{\times})^I$, where $I$ runs over all the nonempty subsets of $\{1,\dots,d_i\}$ and $(\VF^{\times})^I$ means the coordinates corresponding to $\{1,\dots,d_i\}-I$ are removed. Therefore
$$A/(1+\mathfrak m)=\bigsqcup_{\emptyset\not=I_i\subset\{1,\dots,d_i\}, i\in\{1,2\}}\RV^{I_1+I_2}\times\VF^{d_3}.$$
Let $\tilde{f}$ denote the function on $\tilde{X}'_1$ induced by $f$. Since the function $f$ is constant on each orbit $\tilde{\xi}'$, the intersection of $\tilde{X}'_1$ with the piece $\RV^{I_1+I_2}\times\VF^{d_3}$ can be described explicitly as follows
$$\left\{(x,y,z)\in \RV^{I_1+I_2}\times\VF^{d_3}\ 
\begin{array}{|l}
\val_{\rv}(x)\geq 0, \val_{\rv}(y)>0, \val(z)>0\\
\rv(\tilde{f}(x,y,z))=\rv(t)
\end{array}
\right\}.$$
This proves that $\tilde{X}'_1$ belongs to $\Ob\vol\VF^{\bdd}[d_3]$.
\end{proof}

For any $\tilde{\xi}\in\tilde{X}_1$, for any $(x,y,z)\in\tilde{\xi}$, the quantity $\val(x)+\val(y)$ depends only on $\tilde{\xi}$, not on the representative $(x,y,z)$. Thus we can consider the function $\tilde{\lambda}: \tilde{X}_1\to \Gamma_{>0}$ defined by $\tilde{\lambda}(\tilde{\xi})=\val(x)+\val(y)$ if $(x,y,z)$ is in $\tilde{\xi}$. Now by putting $\tilde{X}_{1,\gamma}:=\tilde{\lambda}^{-1}(\gamma)\subset \tilde{X}_1$, the composition of $\tilde{\lambda}$ with $\rho: X_1\to \tilde{X}_1$ yields a definable function $\lambda: X_1\to \Gamma_{>0}$ such that $\tilde{X}_{1,\gamma}$ is the image of $X_{1,\gamma}:=\lambda^{-1}(\gamma)$ under $\rho$. 

The following is a decomposition of $[X_{1,\gamma}]$ which is important in the end of this subsection. Let $\rho_{\gamma}=\rho|_{X_{1,\gamma}}$ and note that, for any $\gamma\in \Gamma_{>0}$, all the fibers of $\rho_{\gamma}^{-1}(\tilde{\xi})$, with $\tilde{\xi}$ varying in $\tilde{X}_{1,\gamma}$, are definably isomorphic to $A_{\gamma}$, where
$$A_{\gamma}:=\left\{u\in\VF \mid 0\leq \val(u)<\gamma\right\},$$ 
because they are annuli of the same modulus $\tilde{\lambda}(\tilde{\xi})=\gamma$. 
Hence the identity
\begin{align}\label{cachan4}
[X_{1,\gamma}]=[\tilde{X}_{1,\gamma}][A_{\gamma}]
\end{align} 
holds in $K(\vol\VF^{\bdd}[*])$.

Let us now consider the function 
$$\nu: \Gamma_{>0}\to \mathscr M_{\ka,\loc}^{\hat{\mu}}$$ 
defined by 
$$\nu(\gamma)=\hl([X_{1,\gamma}]).$$ 
It will be shown that $\nu$ is a definable function in the following sense, and one can take its integral over $\Gamma_{>0}$ in a reasonable way. To have an idea, let us go to a more general situation. Given a definable subset $\Gamma'$ of $\Gamma$, a function $\Gamma'\to \mathscr M_{\ka,\loc}^{\hat{\mu}}$ is called {\it definable} if the source $\Gamma'$ can be divided into finitely many disjoint definable subsets $\Gamma_i$, $i\in I$, such that its restriction to $\Gamma_i$ is constant for every $i\in I$. If $f$ is such a function and $c_i:=f|_{\Gamma_i}\in \mathscr M_{\ka,\loc}^{\hat{\mu}}$, then the integral $\int_{\Gamma'}fd\chi$ is defined as follows
\begin{align}
\int_{\Gamma'}fd\chi:=\sum_{i\in I}c_i\chi(\Gamma_i)\in \mathscr M_{\ka,\loc}^{\hat{\mu}},
\end{align}
where $\chi$ is the o-minimal Euler characteristic defined in \cite[Lem. 9.5]{HK} followed by the localization morphism $\loc$. 

\begin{lemma}\label{cachan6}
The function $\nu$ defined on $\Gamma_{>0}$ by the expression $\nu(\gamma)=\hl([X_{1,\gamma}])$ is a definable function. Moreover, the following identity holds 
\begin{align}\label{cachan3}
\int_{\Gamma_{>0}}\nu d\chi=\hl([X_1]).
\end{align}
\end{lemma}

\begin{proof}
First, let us work with the morphisms of rings $\Psi$ and $\int$ corresponding to (\ref{eq7}) and (\ref{cachan}). We consider $X_1$ as a definable subset of $\VF^d\times\Gamma_{>0}$ by identifying each point $(x,y,z)$ with $(x,y,z,\val(x)+\val(y))$. Then $\int[X_1]\in K(\vol\RV^{\bdd}[d])/I'_{\sp}$ can be presented as follows
$$\int[X_1]=\sum_{\ell=1}^d\Psi([W_{d-\ell}]\otimes [\Delta_{\ell}])+I'_{\sp},$$
where $W_{d-\ell}\subset\RES^{d-\ell}$ and $\Delta_{\ell}\subset\Gamma_{>0}^{\ell}\times\Gamma_{>0}$, and they are bounded. By applying the morphism $\Upsilon$ in Subsection \ref{ups2012}, which is induced from the tensor product of morphism $\beta$ and $\alpha$ thanks to the property that $\ker(\beta\otimes\alpha)$ is contained in $\ker(\Psi)$, one obtains the following 
$$\Upsilon\left(\int[X_1]\right)=\sum_{\ell=1}^d\chi(\Delta_{\ell})[W_{d-\ell}]([\mathbb A_{\ka}^1]-1)^{\ell},$$
living in $!K(\RES)[[\mathbb A_{\ka}^1]^{-1}]_{\loc}$. For simplicity, we put $W'_{d-\ell}:=[W_{d-\ell}]([\mathbb A_{\ka}^1]-1)^{\ell}$, and the above formula is then 
\begin{align}\label{cachan1}
\Upsilon\left(\int[X_1]\right)=\sum_{\ell=1}^d\chi(\Delta_{\ell})[W'_{d-\ell}].
\end{align}

Similarly, by considering $X_{1,\gamma}$ as a definable subset of $\VF^d\times\Gamma_{>0}$, namely $X_{1,\gamma}$ is equal to $\{(x,y,z,\gamma) \mid (x,y,z)\in X_1, \val(x)+\val(y)=\gamma\}$, $\gamma\in\Gamma_{>0}$, we obtain the following 
$$\Upsilon\left(\int[X_{1,\gamma}]\right)=\sum_{\ell=1}^d\chi(\Delta_{\gamma,\ell})[W'_{\gamma,d-\ell}].$$
The elements $\Delta_{\gamma,\ell}$ and $W'_{\gamma,d-\ell}$ are respectively built in the same way as $\Delta_{\ell}$ and $W'_{d-\ell}$ above. Furthermore, although $W'_{d-\ell}$ and $W'_{\gamma,d-\ell}$ are different, their class in $!K(\RES)$ is the same thing. Indeed, $W'_{\gamma,d-\ell}$ is exactly $W'_{d-\ell}$ augmented by one condition on $\val$ (which comes from $\val(x)+\val(y)=\gamma$), but ``$!$'' means that such a difference on the $\val$-function may be ignored, and hence
\begin{align}\label{cachan2}
\Upsilon\left(\int[X_{1,\gamma}]\right)=\sum_{\ell=1}^d\chi(\Delta_{\gamma,\ell})[W'_{d-\ell}].
\end{align}

Now observe that, for any $\ell=1,\dots,d$, all $\Delta_{\gamma,\ell}$, with $\gamma$ running over some subset $\Gamma(\ell)\subset\Gamma_{>0}$, are in definable bijection. There is no bijection between $\Delta_{\gamma,\ell}$ and $\Delta_{\gamma',\ell'}$ if $\ell\not=\ell'$, as they are in different dimensions. We consider the following second projection
$$pr_2: \Gamma_{>0}^d\times \Gamma_{>0}\to \Gamma_{>0}.$$
For any $\ell=1,\dots,d$, if one identifies $\Gamma_{>0}^{\ell}\times \Gamma_{>0}$ in a natural way with a definable subset of $\Gamma_{>0}^d\times \Gamma_{>0}$ and if let $\Gamma_{\ell}$ be the image of $\Delta_{\ell}\subset \Gamma_{>0}^{\ell}\times \Gamma_{>0}$ through $pr_2$, then $\Gamma_{>0}$ is equal to the disjoint union $\bigsqcup_{\ell=1}^d\Gamma_{\ell}$. One also remarks that in fact $\Gamma_{\ell}=\Gamma(\ell)$ for $\ell=1,\dots,d$. In other words, we have deduced from (\ref{cachan2}) that the function $\nu(\gamma)=\hl([x_{1,\gamma}])$ is the constant $c_{\ell}\tilde{\Theta}\Big([W'_{d-\ell}]\Big)$, with $c_{\ell}:=\chi(\Delta_{\gamma,\ell})$, on each subset $\Gamma_{\ell}$, $\ell=1,\dots,d$, which proves the definability of $\nu$. Moreover, one has
$$\int_{\Gamma_{>0}}\nu d\chi= \sum_{\ell=1}^dc_{\ell}\chi(\Gamma_{\ell})\tilde{\Theta}\Big([W'_{d-\ell}]\Big),$$
where by the previous arguments one deduces that $\chi(\Delta_{\ell})=c_{\ell}\chi(\Gamma_{\ell})$, $\ell=1,\dots,d$. This together with (\ref{cachan1}) implies (\ref{cachan3}).
\end{proof}

Now thanks to (\ref{cachan4}) one deduces that $\nu(\gamma)=\hl([\tilde{X}_{1,\gamma}])\hl([A_{\gamma}])$ for $\gamma\in\Gamma_{>0}$. As proved in Lemma \ref{cachan5} below, $\hl([A_{\gamma}])=0$ for all $\gamma\in\Gamma_{>0}$, thus $\nu(\gamma)=0$ for all $\gamma\in\Gamma_{>0}$. This together with Lemma \ref{cachan6} implies $\hl([X_1])=0$.

\begin{lemma}\label{cachan5}
For any $\gamma\in\Gamma_{>0}$, $\hl([A_{\gamma}])=0$.
\end{lemma}

\begin{proof}
First, we note that $A_{\gamma}=\{u\in\VF \mid 0\leq \val(u)<\gamma\}$ and it satisfies the assumption of Proposition \ref{eq010}; namely, $A_{\gamma}$ is bounded, $\ka((t))$-definable in $\VF$ and 2-invariant. By writing $\gamma=p/q$ with $(p,q)=1$, we have the following:

If $n\in \mathbb N_{>0}$ is not divisible by $q$, then $\hl_n([A_{p/q}])=0$.

If $n=mq$, $m\in \mathbb N_{>0}$, then, by Proposition \ref{eq010},
\begin{align*}
\HL_{mq}([A_{p/q}])&=\left[\left\{\varphi(t)\in \ka[t]/t^{2mq} \mid 0\leq \ord_t\varphi(t)<mp\right\}\right][\mathbb A_{\ka}^1]^{-2mq+1}\\
&=[\mathbb A_{\ka}^1]-[\mathbb A_{\ka}^1]^{-mr+1},
\end{align*}
where $r=\min\{p,2q\}$. Therefore,
\begin{align*}
\hl([A_{p/q}])&=-\lim_{T\to\infty}\sum_{m\geq 1}\hl_{mq}([A_{p/q}])T^{mq}\\
&=-[\mathbb A_{\ka}^1]\lim_{T\to\infty}\left(\frac{T^q}{1-T^q}-\frac{[\mathbb A_{\ka}^1]^{-r}T^q}{1-[\mathbb A_{\ka}^1]^{-r}T^q}\right)\\
&=0.
\end{align*}
This proves the lemma.
\end{proof}


\section{Motivic integration on rigid varieties}\label{sec6}
Let $K$ be a non-archimedean complete discretely valued field of equal characteristics zero, with valuation ring $R$ and residue field $\ka$. We fix a uniformizing parameter $\varpi$ in $R$ 

\subsection{Motivic integration for a gauge form}
If $\mathfrak X$ is a separated generically smooth formal $R$-scheme topologically of finite type, motivic integration on $\mathfrak X$ was constructed by Sebag \cite{Se} and enriched by Loeser-Sebag \cite{LS} and Nicaise-Sebag \cite{NS, NS1, NS2}. We can refer to \cite[Subsec. 6.1]{NS} for definition of motivic integration of a gauge form $\omega$ on $\mathfrak X_{\eta}$, which takes values in $\mathscr M_{\mathfrak X_0}$, denoted by $\int_{\mathfrak X}|\omega|$. 

A more direct way to define $\int_{\mathfrak X}|\omega|$ is as follows. Let us consider $\mathfrak X$ as the inductive limit of the $R_m$-schemes topologically of finite type $\mathscr X_m=(\mathfrak X, \mathcal O_{\mathfrak X}\otimes_RR_m)$ in the category of formal $R$-schemes, where $R_m$ denotes $R/(\varpi)^{m+1}$. By Greenberg \cite{Gr}, the functor $\mathscr Y\mapsto \Hom_{R_m}(\mathscr Y\times_{\ka}R_m,\mathscr X_m)$ from the category of $\ka$-schemes to the category of sets is presented by a $\ka$-scheme $\Gr_m(\mathscr X_m)$ topologically of finite type such that, for every $\ka$-algebra $\mathcal A$, $\Gr_m(\mathscr X_m)(\mathcal A)=\mathscr X_m(\mathcal A\otimes_{\ka}R_m)$. Then one can take the projective limit $\Gr(\mathfrak X)$ of the system $(\Gr_m(\mathscr X_m))_{m\in \mathbb N}$ in the category of $\ka$-schemes (cf. \cite{Gr}, \cite{Se}, \cite{LS}). It was proved in \cite{Gr} that the functor {\rm $\Gr$} respects open and closed immersions and fiber products, and sends affine topologically of finite type formal $R$-schemes to affine $\ka$-schemes. Notice that the notions of piecewise trivial fibration, cylindrical stable subset of $\Gr(\mathfrak X)$ in this context were already introduced in \cite{Se} and \cite{LS}, in which the latter requires in addition $\mathfrak X$ flat. Let $\mathbf C_{0,\mathfrak X}$ be the set of stable cylindrical subsets of $\Gr(\mathfrak X)$ of some level. 

\begin{proposition}\label{stmande}
There exists a unique additive morphism $\tilde{\mu}: \mathbf C_{0,\mathfrak X}\to \mathscr M_{\mathfrak X_0}$ given by $\tilde{\mu}(A)=[\pi_m(A)][\mathbb A_{\mathfrak X_0}^1]^{-(m+1)d}$ for $A$ a stable cylinder of level $m$, where $d$ is the relative dimension of $\mathfrak X$ and $\pi_m$ is the canonical projection $\Gr(\mathfrak X)\to\Gr_m(\mathscr X_m)$. 
\end{proposition}
\begin{proof}
By \cite[Lem. 4.3.25]{Se}, one has $[\pi_n(A)]=[\pi_m(A)][\mathbb A_{\mathfrak X_0}^1]^{(n-m)d}$, where $[\pi_m(A)]$ is the class of $\pi_m(A)\to\mathfrak X_0$ in $\mathscr M_{\mathfrak X_0}$. Thus
$$[\pi_n(A)][\mathbb A_{\mathfrak X_0}^1]^{-(n+1)d}=[\pi_m(A)][\mathbb A_{\mathfrak X_0}^1]^{-(m+1)d}$$
for any $n\geq m$. Because $A$ is a stable cylinder of level $m$, it is stable of level $n\geq m$, thus $\tilde{\mu}$ is well defined. The additivity of $\tilde{\mu}$ is obvious by definition.
\end{proof}

Given a cylinder $A$ in $\mathbf C_{0,\mathfrak X}$ and a function $\alpha: A\to \mathbb Z \cup \{\infty\}$ taking only a finite number of values with $\alpha^{-1}(m)$ in $\mathbf C_{0,\mathfrak X}$ for all $m$. Then one puts
$$\int_A[\mathbb A_{\mathfrak X_0}^1]^{-\alpha}d\tilde{\mu}:=\sum_{m\in\mathbb Z}\tilde{\mu}(\alpha^{-1}(m))[\mathbb A_{\mathfrak X_0}^1]^{-m}.$$
If $\omega$ is a gauge form on $\mathfrak X_{\eta}$, it admits by \cite{LS} an integer-valued function $\ord_{\varpi, \mathfrak X}(\omega)$ on $\Gr(\mathfrak X)$ satisfying the previous properties as $\alpha$. We are now able to define $\int_{\mathfrak X}|\omega|$ to be the integral $\int_{\Gr(\mathfrak X)}[\mathbb A_{\mathfrak X_0}^1]^{-\ord_{\varpi, \mathfrak X}(\omega)}d\tilde{\mu}$, which belongs to $\mathscr M_{\mathfrak X_0}$.

The image of $\int_{\mathfrak X}|\omega|$ under the forgetful morphism $\mathscr M_{\mathfrak X_0}\to \mathscr M_{\ka}$ only depends on $\mathfrak X_{\eta}$, not on $\mathfrak X$, and it is exactly the motivic integral $\int_{\mathfrak X_{\eta}}|\omega|$ defined in \cite[Thm.-Def. 4.1.2]{LS}. We shall also denote in this article the image by $\int_{\mathfrak X_{\eta}}|\omega|$.

\begin{remark}\label{wedding}
If $\mathfrak X$ is a separated generically smooth formal $R$-scheme topologically of finite type, then $\mathfrak X_{\eta}$ is a separated quasi-compact smooth rigid $K$-variety. In fact, Loeser-Sebag in \cite[Thm.-Def. 4.1.2]{LS} defined $\int_X|\omega|$ for any quasi-compact smooth rigid $K$-variety $X$ and any differential form of maximal degree $\omega$ on $X$.
\end{remark}

Generically smooth special formal $R$-schemes have been considered in \cite{NS2}. The theory of motivic integration has been systematically developed in that context by Nicaise in \cite{Ni2}. Firstly, let $\mathfrak X$ be a special formal $R$-scheme, not necessarily generically smooth. By a {\it N\'eron smoothening} for $\mathfrak X$ we mean a morphism of special formal $R$-schemes $\mathfrak Y \to\mathfrak X$, with $\mathfrak Y$ adic smooth over $R$, which induces an open embedding $\mathfrak Y_{\eta} \to \mathfrak X_{\eta}$ satisfying $\mathfrak Y_{\eta}(\tilde{K})=\mathfrak X_{\eta}(\tilde{K})$ for any finite unramified extension $\tilde{K}$ of $K$. By \cite{Ni2}, if in addition $\mathfrak X$ is generically smooth, it obviously admits a N\'eron smoothening $\mathfrak Y\to\mathfrak X$, and in this case one can even choose $\mathfrak Y$ to be a separated generically smooth formal $R$-scheme topologically of finite type. Keeping the situation, due to \cite[Prop. 4.7, 4.8]{Ni2}, one defines 
$$\int_{\mathfrak X}|\omega|:= \int_{\mathfrak Y}|\omega|, \ \text{and}\ \int_{\mathfrak X_{\eta}}|\omega|:= \int_{\mathfrak Y_{\eta}}|\omega|$$
for a gauge form $\omega$ on $\mathfrak X_{\eta}$. Of course, these are well defined, also by \cite{Ni2}. Note that the integral $\int_{\mathfrak X}|\omega|$ lives in $\mathscr M_{\mathfrak X_0}$, while $\int_{\mathfrak X_{\eta}}|\omega|$ belongs to $\mathscr M_{\ka}$.

\subsection{Bounded smooth rigid varieties}
If $\mathfrak X$ is a special formal $R$-scheme, the generic fiber $\mathfrak X_{\eta}$ is a {\it bounded} rigid $K$-variety. According to Nicaise-Sebag \cite{NS2}, generally, a rigid $K$-variety $X$ is bounded if there exists a quasi-compact open subspace $Y$ of $X$ such that $Y(\tilde{K})=X(\tilde{K})$ for any finite unramified extension $\tilde{K}$ of $K$. If the rigid variety $X$ is smooth, so is $Y$. The motivic integration on a quasi-compact smooth rigid variety was already defined by Loeser-Sebag \cite{LS} (also mentioned in Remark \ref{wedding}), and inspired by this, Nicaise-Sebag \cite{NS2} extend the notion to bounded smooth rigid $K$-varieties.  Their definition is as follows 
$$\int_X|\omega|:=\int_Y|\omega|\ \in \mathscr M_{\ka},$$ 
where $\omega$ is a gauge form on $X$. The integral is well defined, due to \cite[Prop. 5.9]{NS2}.

Given a natural $d$. Let $\GBSRig_K^d$ be the category of gauged bounded smooth rigid $K$-varieties of dimension $d$, i.e., each object of $\GBSRig_K^d$ is a pair $(X,\omega)$ with $X$ a bounded smooth rigid $K$-variety of dimension $d$ and $\omega$ a gauge form on $X$, and each morphism $h:(X',\omega')\to (X,\omega)$ in $\GBSRig_K^d$ is a morphism of bounded smooth rigid $K$-varieties $h:X'\to X$ such that $h^*\omega=\omega'$. The Grothendieck group $K(\GBSRig_K^d)$ is the quotient of the free abelian group generated by symbols $[X,\omega]$ with $(X,\omega)$ an object of $\Ob\GBSRig_K^d$ by the relations $[X',\omega']=[X,\omega]$ if $(X',\omega')\cong (X,\omega)$ in $\GBSRig_K^d$, and 
$$[X,\omega]=\sum_{\emptyset\not=I\subset J}(-1)^{|I|-1}[O_I,\omega|_{O_I}],$$
whenever $(O_i)_{i\in J}$ is a finite admissible covering of $X$, $O_I=\bigcap_{i\in I}O_i$ for any $I\subset J$. One puts 
$$K(\GBSRig_K):=\bigoplus_{d\geq 0}K(\GBSRig_K^d)$$
and defines a product on it as follows $[X,\omega]\cdot[X',\omega']:=[X\times X',\omega\times\omega']$. Together with this product, the Grothendieck group $K(\GBSRig_K)$ becomes a ring.

\begin{proposition}\label{gsr}
There is a unique morphism of rings $\Phi: K(\GBSRig_K)\to \mathscr M_{\ka}$ such that $\Phi([X,\omega])=\int_X|\omega|$.
\end{proposition}

\begin{proof}
For $[X,\omega]$ in $K(\GBSRig_K)$, we set $\Phi([X,\omega])=\int_X|\omega|$. The additivity of $\Phi$ is clear by the definition of $\int_X|\omega|$ and \cite[Prop. 4.2.1]{LS}. Let $[X,\omega]$ and $[X',\omega']$ be in $K(\GBSRig_K)$. If $Y$ and $Y'$ are quasi-compact smooth rigid $K$-varieties such that $Y(\tilde{K})=X(\tilde{K})$ and $Y(\tilde{K})=X(\tilde{K})$ for any finite unramified extension $\tilde{K}$ of $K$, then so is $(Y\times Y')(\tilde{K})=(X\times X')(\tilde{K})$, since $(Y\times Y')(\tilde{K})=Y(\tilde{K})\times Y'(\tilde{K})$ and $(X\times X')(\tilde{K})=X(\tilde{K})\times X'(\tilde{K})$. We deduce from definition and \cite[Prop. 4.2.1]{LS} that 
\begin{align*}
\int_{X\times X'}|\omega\times \omega'|=\int_{Y\times Y'}|\omega\times \omega'|=\int_{Y}|\omega|\cdot \int_{Y'}|\omega'|=\int_X|\omega|\cdot \int_{X'}|\omega'|.
\end{align*}
This finishes the proof.
\end{proof}


\subsection{Motivic volume}\label{tetdenroi}
For $m\geq 1$, let $K(m):=K[T]/(T^m-\varpi)$ be a totally ramified extension of degree $m$ of $K$, and $R(m):=R[T]/(T^m-\varpi)$ the normalization of $R$ in $K(m)$. If $\mathfrak X$ is a formal $R$-scheme, we define $\mathfrak X(m):=\mathfrak X\times_RR(m)$ and $\mathfrak X_{\eta}(m):=\mathfrak X_{\eta}\times_KK(m)$. If $\omega$ is a differential form on $\mathfrak X_{\eta}$, we denote by $\omega(m)$ the pullback of $\omega$ via the natural morphism $\mathfrak X_{\eta}(m)\to \mathfrak X_{\eta}$.

Let $\mathfrak X$ be a generically smooth special formal $R$-scheme, $\omega$ a gauge form on $\mathfrak X_{\eta}$. The volume Poincar\'e series of $(\mathfrak X,\omega)$ is defined to be an element of $\mathscr M_{\mathfrak X_0}[[T]]$ as
\begin{align}\label{eq103}
S(\mathfrak X,\omega; T):=\sum_{m\geq 1}\left(\int_{\mathfrak X(m)}|\omega(m)|\right)T^m.
\end{align}
This power series has been introduced and studied first in \cite{NS} in the context of generically smooth separated formal schemes topologically of finite type over $R$; in \cite{Ni2} its study has been extended in the framework of generically smooth special formal $R$-schemes.

\begin{remark}
In general, by \cite[Rmk. 4.10]{Ni2}, the volume Poincar\'e series $S(\mathfrak X,\omega; T)$ depends on the choice of $\varpi$, i.e., on the $K$-fields $K(m)$. In the case where $\ka$ is algebraically closed, however, $K(m)$ is the unique extension of degree $m$ of $K$, up to $K$-isomorphism, thus $S(\mathfrak X,\omega; T)$ is independent of the choice of $\varpi$. 

For simplicity (and also in our case), we prove this for $R=\ka[[t]]$ and $K=\ka((t))$. If $t'$ is another uniformizing parameter for $\ka[[t]]$, then $t'=\alpha t$ with $\alpha=\alpha(t)\in\ka[[t]]$ and $\alpha(0)\in\ka^{\times}$. Since $\ka$ is algebraically closed, all the $m$th roots of an element of $\ka$ belong to $\ka$, thus all the $m$th roots of $\alpha$ are in $\ka[[t]]$. Pick one among them and denote this element by $\alpha^{1/m}$. The correspondence $t^{1/m}\mapsto \alpha^{1/m}t^{1/m}$ defines a canonical isomorphism of $\ka((t))$-fields $\ka((t^{1/m}))\to \ka((t^{\prime 1/m}))$. Thus, for each $m\in\mathbb N_{>0}$, $\int_{\mathfrak X(m)}|\omega(m)|$ (and hence $S(\mathfrak X,\omega; T)$) is independent of the choice of the uniformizing parameter $t$.
\end{remark}

Using resolution of singularities, Nicaise \cite[Cor. 7.13]{Ni2} proved that, with the previous hypotheses, and in addition the $\omega$ being $\mathfrak X$-bounded (see \cite[Def. 2.11]{Ni2} for definition), the volume Poincar\'e series of $(\mathfrak X,\omega)$ is a rational series. In this case, the limit 
$$S(\mathfrak X,\widehat{K^s}):=-\lim_{T\to\infty}S(\mathfrak X,\omega;T)$$ 
is called the {\it motivic volume} of $\mathfrak X$, and its image under the forgetful morphism, $S(\mathfrak X_{\eta},\widehat{K^s})$, which obviously admits the indentity
$$S(\mathfrak X_{\eta},\widehat{K^s})=-\lim_{T\to \infty}\sum_{m\geq 1}\Phi\Big([\mathfrak X_{\eta}(m),\omega(m)]\Big)T^m,$$
is called the {\it motivic volume} of $\mathfrak X_{\eta}$. The latter is a quantity in $\mathscr M_{\ka}$.

Let $K(\BSRig_K)$ be the Grothendieck ring of the category $\BSRig_K$ of bounded smooth rigid $K$-varieties. It obtains from $K(\GBSRig_K)$ in Proposition \ref{gsr} by forgetting gauge forms.

\begin{proposition}\label{cor3}
There exists a homomorphism of additive groups {\rm 
$$\SE: K(\BSRig_K)\to \mathscr M_{\ka}$$}such that {\rm $\SE([\mathfrak X_{\eta}])=S(\mathfrak X_{\eta},\widehat{K^s})$} for $\mathfrak X$ a generically smooth special formal $R$-scheme. 
\end{proposition}
\begin{proof}
The morphism $\SE$ is defined by 
$$\SE([X])=-\lim_{T\to \infty}\sum_{m\geq 1}\Phi\Big([X(m),\omega(m)]\Big)T^m,$$
where $\omega$ is an $\mathfrak X$-bounded gauge form on $X$ with $\mathfrak X$ a formal $R$-model of $X$. This generically smooth model $\mathfrak X$ can be chosen to be a special formal $R$-scheme, thus the limit of the series on the right exists (induced from \cite[Cor. 7.13]{Ni2}). That $\SE$ is a morphism of groups follows from the property of $\Phi$.
\end{proof}

\subsection{Resolution of singularities}\label{mncs}
In what follows, we shall work over $K=\ka((t))$ and $R=\ka[[t]]$. The notions of motivic nearby cycles and motivic Milnor fiber were introduced from the beginning without explaining why they are well defined. We shall discuss this problem in the present subsection.

Let $\mathfrak X$ be a generically smooth special formal $R$-scheme of relative dimension $d$, and let $\mathfrak h: \mathfrak Y \to \mathfrak X$ be a resolution of singularities of $(\mathfrak X,\mathfrak X_0)$ (the definition and the existence are referred to \cite{Tem}, or \cite{Ni2}). Assume that the divisor $\mathfrak Y_s$ is written as $\sum_{i\in J}N_i\mathfrak E_i$, where $\mathfrak E_i$, with $i\in J$, are the irreducible components of $\mathfrak Y_s=\mathfrak h^{-1}(\mathfrak X_s)$. Denoting $E_i=(\mathfrak E_i)_0$ for $i\in J$, for nonempty $I\subset J$, one puts 
$$E_I:=\bigcap_{i\in I}E_i,\quad E_I^{\circ}:=E_I\setminus\bigcup_{j\not\in I}E_j.$$ 
Let $\{U\}$ be a covering of $\mathfrak Y$ by affine open subschemes such that $U\cap E_I^{\circ}\not=\emptyset$, and on this set, $\mathfrak f\circ\mathfrak h= \tilde{u}\prod_{i\in I}y_i^{N_i}$, with $\tilde{u}$ a unit on $U$ and $y_i$ a local coordinate defining $E_i$. Set $m_I:=\gcd(N_i)_{i\in I}$. Similarly as in \cite{DL5}, there is an unramified Galois covering $\pi_I:\tilde{E}_I^{\circ}\to E_I^{\circ}$ with Galois group $\mu_{m_I}$ given over $U\cap E_I^{\circ}$ by 
$$\left\{(z,y)\in \mathbb{A}_{\ka}^1\times(U\cap E_I^{\circ}) \mid z^{m_I}=\tilde{u}(y)^{-1}\right\}.$$
The covering is endowed with a natural $\mu_{m_I}$-action good over $E_I^{\circ}$ obtained by multiplying the $z$-coordinate with elements of $\mu_{m_I}$, thus the $\hat{\mu}$-equivariant morphism $\mathfrak h\circ\pi_I$ defines a class $[\tilde{E}_I^{\circ}]$ in $\mathscr{M}_{\mathfrak X_0}^{\hat{\mu}}$. 

The following is based on Lemma \ref{jussieu17may} and inspired by \cite{KS}.

\begin{definition}\label{minhy}
One defines the {\it motivic nearby cycles} $\mathcal S_{\mathfrak f}$ of the formal function $\mathfrak f$ as an element of $\mathscr M_{\mathfrak X_0}^{\hat{\mu}}$ by 
\begin{align*}
\mathcal S_{\mathfrak f}:=\sum_{\emptyset\not=I\subset J} (1-[\mathbb A_{\mathfrak X_0}^1])^{|I|-1}[\tilde{E_I^{\circ}}].
\end{align*}
\end{definition}

\begin{lemma}\label{jussieu17may}
$\mathcal S_{\mathfrak f}$ is independent of the resolution of singularities $\mathfrak h$. 
\end{lemma}
\begin{proof}
This is true without $\hat{\mu}$-action because 
$S(\mathfrak X,\widehat{K^s})=[\mathbb A_{\mathfrak X_0}^1]^{-d}\mathcal S_{\mathfrak f}$ in $\mathscr M_{\mathfrak X_0}$, due to Definition \ref{minhy} and \cite[Prop. 7.36]{Ni2}. Studying the volume Poincar\'e series (\ref{eq103}), we assume in addition that $\omega$ is $\mathfrak X$-bounded (cf. \cite[Def. 2.11]{Ni2}). Then by \cite[Thm. 7.12]{Ni2}, there exist natural numbers $\alpha_i\geq 1$, $i\in J$, such that $\int_{\mathfrak X(m)}|\omega(m)|$ is equal in $\mathscr M_{\mathfrak X_0}$ to
\begin{equation}\label{demdong}
[\mathbb A_{\mathfrak X_0}^1]^{-d}\sum_{\emptyset\not=I\subset J}([\mathbb A_{\mathfrak X_0}^1]-1)^{|I|-1}[\tilde{E}_I^{\circ}]\left(\sum_{\begin{smallmatrix} k_i\geq 1, i\in I\\ \sum_{i\in I}k_iN_i=m \end{smallmatrix}}[\mathbb A_{\mathfrak X_0}^1]^{-\sum_{i\in I}k_i\alpha_i}\right)
\end{equation}
In what follows we are going to define a canonical $\mu_m$-action by which $\int_{\mathfrak X(m)}|\omega(m)|$ belongs to $\mathscr M_{\mathfrak X_0}^{\hat{\mu}}$. Let $\mathfrak W_m\to \mathfrak X(m)$ be a N\'eron smoothening for $\mathfrak X(m)$ such that $\mathfrak W_m$ is a separated generically smooth formal $\ka[[t^{1/m}]]$-scheme topologically of finite type. Then, one has 
$$\int_{\mathfrak X(m)}|\omega(m)|=\int_{\mathfrak W_m}|\omega(m)|=\int_{\Gr(\mathfrak W_m)}[\mathbb A_{\mathfrak X_0}^1]^{-\ord_{t^{1/m}, \mathfrak W_m}(\omega(m))}d\tilde{\mu}.$$
The canonical $\mu_m$-action on $\Gr(\mathfrak W_m)$ is defined as follows: $a\varphi(t^{1/m})=\varphi(at^{1/m})$. By this, in the spirit of \cite[Thm. 3.2]{DL5}, $\int_{\mathfrak X(m)}|\omega(m)|$ must be equal to (\ref{demdong}) in $\mathscr M_{\mathfrak X_0}^{\hat{\mu}}$, hence (\ref{demdong}) is independent of the choice of $h$. It thus implies that the quantity $S(\mathfrak X,\widehat{K^s})=[\mathbb A_{\mathfrak X_0}^1]^{-d}\mathcal S_{\mathfrak f}$ is independent of the choice of $h$ when considered as an element of $\mathscr M_{\mathfrak X_0}^{\hat{\mu}}$.
\end{proof}

Let $\int_{\mathfrak X_0}$ be the forgetful morphism $\mathscr M_{\mathfrak X_0}\to \mathscr M_{\ka}$. Lemma \ref{jussieu17may} and its proof, together with Proposition \ref{cor3}, provide the below important corollary. 
\begin{corollary}\label{sosanh}
If $\mathfrak X$ is a generically smooth special formal $R$-scheme of relative dimension $d$, then $\SE([\mathfrak X_{\eta}])=[\mathbb A_{\ka}^1]^{-d}\int_{\mathfrak X_0}\mathcal S_{\mathfrak f}$ in $\mathscr M_{\kappa}^{\hat{\mu}}$. As a consequence, $\SE$ is also a morphism of groups of {\rm $K(\BSRig_K)$} into {\rm $\mathscr M_{\ka}^{\hat{\mu}}$}.
\end{corollary}


\section{The case of an algebraically closed field}\label{sec66}
Let us remark that the theory $\ACVF(0,0)$ is valuable for the rigid varieties in the present article. 

\subsection{Generalized measured categories}
In this section, we are interested in volume forms that are more general than those in Section \ref{prepare1}, and \cite{HK, HK2} are still main references. Indeed, first of all, we consider the category $\mu_{\Gamma}\VF[*]$ of $A$-definable sets with definable volume forms up to $\Gamma$-equivalence. Notice that, for simplicity, we shall write from now on ``definable'' to mean ``$A$-definable'' when $A$ is already clear. Let $n$ be a positive natural number. An object in $\mu_{\Gamma}\VF[n]$ is a triple $(X,f,\alpha)$ with $X$ a definable subset of $\VF^k\times \RV^{\ell}$, for some $k,\ell\geq 0$, $f:X\to \VF^n$ a definable map with finite fibers and $\alpha:X\to \Gamma$ a definable function; a morphism $(X,f,\alpha)\to (X',f',\alpha')$ of its is a definable essential bijection $F:X\to X'$ such that for almost every $x\in X$, 
$$\alpha(x) =\alpha'(F(x))+\val(\Jac F(x)).$$ 
Let $\mu_{\Gamma}\VF^{\bdd}[n]$ be the full subcategory of $\mu_{\Gamma}\VF[n]$ such that its objects are bounded definable sets with bounded definable forms $\alpha$. As previous, we can define $\mu_{\Gamma}\VF[*]=\bigoplus_{n\geq 0}\mu_{\Gamma}\VF[n]$ and $\mu_{\Gamma}\VF^{\bdd}[*]=\bigoplus_{n\geq 0}\mu_{\Gamma}\VF^{\bdd}[n]$. In particular, the category $\vol\VF[*]$ (resp. $\vol\VF^{\bdd}[*]$) in Section \ref{prepare1} is the full subcategory of $\mu_{\Gamma}\VF[*]$ (resp. $\mu_{\Gamma}\VF^{\bdd}[*]$) consisting of the objects of the form $(X,f,0)$.

Also following \cite{HK}, objects of $\mu_{\Gamma}\RV[n]$ are triples $(X,f,\alpha)$ with $X$ a definable subset of $\RV^{k}$, for some $k\geq 0$, $f:X\to\RV^n$ a definable map with finite fibers, and $\alpha: X\to \Gamma$ a definable function. A morphism $(X,f,\alpha)\to (X',f',\alpha')$ is a definable bijection $F:X\to X'$ such that for all $x\in X$, 
$$\alpha(x)+\sum_{i=1}^n\val_{\rv}f_i(x)= \alpha'(F(x))+\sum_{i=1}^n\val_{\rv}f'_i(F(x)).$$ 
The category $\mu_{\Gamma}\RES[n]$ is by definition the full subcategory of $\mu_{\Gamma}\RV[n]$ such that, for each object $(X,f,\alpha)$, we have that $\val_{\rv}(X)$ is a finite set. The category $\mu_{\Gamma}\RV^{\bdd}[n]$ is defined in the same way as previous. For simplicity, we sometimes omit the notation of structural map $f$ in the triple $(X,f,\alpha)$ when it is already clear. Define $\mu_{\Gamma}\RV[*]=\bigoplus_{n\geq 0}\mu_{\Gamma}\RV[n]$ and $\mu_{\Gamma}\RV^{\bdd}[*]=\bigoplus_{n\geq 0}\mu_{\Gamma}\RV^{\bdd}[n]$.

By \cite{HK}, the category $\mu\Gamma[n]$ consists of objects which are pairs $(\Delta,l)$, where $\Delta\in\Ob\Gamma[n]$ and $l:\Delta\to\Gamma$ is a definable map. A morphism $(\Delta,l)\to (\Delta',l')$ of its is a definable bijection $h:\Delta\to\Delta'$ which is liftable to a definable bijection $\val_{\rv}^{-1}\Delta\to \val_{\rv}^{-1}\Delta'$ such that 
$$|x|+l(x)=|h(x)|+l'(h(x)).$$ In particular, $\vol\Gamma[n]$ is the full subcategory of $\mu\Gamma[n]$ defined as above with $l=0$. The category $\mu\Gamma^{\bdd}[n]$ is the full subcategory of $\mu\Gamma[n]$ such that, for each $\Delta$ in $\Ob\mu\Gamma^{\bdd}[n]$, there exists a $\gamma\in\Gamma$ with $\Delta\subset [\gamma,\infty)^n$. The category $\mu\Gamma[*]$ (resp. $\mu\Gamma^{\bdd}[*]$) is the direct sum $\bigoplus_{n\geq 1}\mu\Gamma[n]$ (resp. $\bigoplus_{n\geq 1}\mu\Gamma^{\bdd}[n]$).


\subsection{Morphisms between Grothendieck rings}
The lifting map 
$$\L: \Ob\mu_{\Gamma}\RV[n]\to\Ob\mu_{\Gamma}\VF[n]$$ 
is defined as follows. For any $(X,f,\alpha)$ in $\Ob\mu_{\Gamma}\RV[n]$, we put 
$$\L(X,f,\alpha)=(\L X,\L f,\L\alpha),$$ 
where $\L X=X\times_{f,\rv}(\VF^{\times})^n$, $\L f(a,b)=f(a,\rv(b))$ and $\L\alpha (a,b)=\alpha(a,\rv(b))$. This map induces a canonical morphism of semi-rings  
$$\int: K_+(\mu_{\Gamma}\VF^{\bdd}[n])\to K_+(\mu_{\Gamma}\RV^{\bdd}[n])/I'_{\sp},$$
where $I'_{\sp}$ is the congruence generated by $[1]_1=[\RV^{>0}]_1$ with the constant volume form $0$ in $\Gamma$. We also write $\int$ for the induced morphism between the corresponding rings.

Let $\mu\Gamma^{\fin}[*]$ be the full subcategory of $\mu\Gamma[*]$ whose objects are finite sets. There exists a natural map of Grothendieck semi-rings 
$$K_+(\mu_{\Gamma}\RES)\otimes_{\mu\Gamma^{\fin}[*]} K_+\mu\Gamma^{\bdd}[*])\to K_+(\mu_{\Gamma}\RV^{\bdd}[*]).
$$
Namely, it is built from the morphisms $K_+(\mu_{\Gamma}\RES)\to K_+(\mu_{\Gamma}\RV^{\bdd}[*])$ (induced by the inclusion) and $K_+\mu\Gamma^{\bdd}[*])\to K_+(\mu_{\Gamma}\RV^{\bdd}[*])$ (defined by $\Delta\mapsto \val_{\rv}^{-1}(\Delta)$). By \cite[Prop. 10.10]{HK}, this natural morphism is an isomorphism of rings. Thus an element of $K_+(\mu_{\Gamma}\RV^{\bdd}[*])$ can be written as a finite sum 
$$\sum [(X\times\val_{\rv}^{-1}(\Delta),f,\alpha)].$$
An argument in the proof of \cite[Prop. 10.10]{HK} also shows that 
$$[(X\times\val_{\rv}^{-1}(\Delta),f,\alpha)]=[(X,f_0,1)]\otimes [(\Delta,l)],$$
where $f_0:X\to \RV^n$ and $l:\Delta\to\Gamma$ are some definable functions. By this, in order to construct morphisms from an appropriate subring of $K(\mu_{\Gamma}\RV^{\bdd}[*])$ to $!K(\RES)[[\mathbb A_{\ka}^1]^{-1}]_{\loc}$ similarly as in Subsection \ref{hm2012}, it suffices to define $a_m(\Delta,l)$ and $b_m(X,f,1)$.

Let $\mu_{\Gamma}^G\RV[*]$ be the full subcategory of $\mu_{\Gamma}\RV^{\bdd}[*]$ consisting of objects $(X,f,\alpha)$ such that $\alpha(X)$ is two-sided bounded in $\Gamma$. Clearly, this subcategory contains $\vol\RV^{\bdd}$ since we can take $\alpha:X\to\Gamma$ to be the constant function $0$. Let $\mu^G\Gamma[*]$ be the full subcategory of $\mu\Gamma^{\bdd}[*]$ whose objects $(\Delta,l)$ have the property that $l(\Delta)$ is two-sided bounded in $\Gamma$. There is in the same way as (\ref{eq7}) a canonical morphism of semi-rings
$$K_+(\mu_{\Gamma}\RES[*])\otimes K_+(\mu^G\Gamma[*])\to K_+(\mu_{\Gamma}^G\RV[*])$$
whose kernel is generated by $[\val_{\rv}^{-1}(\gamma)]_1\otimes 1 - 1\otimes [\gamma]_1$, with $\gamma$ definable in $\Gamma$. Let $m\in\mathbb N_{>0}$, $e\in \Gamma$ and $(\Delta,l)\in \Ob\mu^G\Gamma[n]$. We define $\Delta(m):=\Delta\cap(1/m \mathbb Z)^n$ and $\Delta_e:=l^{-1}(e)$. Put 
\begin{align}\label{0eq0001}
\tilde a_m(\Delta,l)=\sum_{e\in \Gamma}\sum_{\gamma\in \Delta_e(m)}[\mathbb A_{\ka}^1]^{-m(|\gamma|+e)}([\mathbb A_{\ka}^1]-1)^n,
\end{align}
where each $e$-term is equal to zero if $me\not\in\mathbb N$. Since $l(\Delta)$ is two-sided bounded in $\Gamma$, the index $e$ of the previous sum runs over a finite set. Thus, similarly as in Subsection \ref{hm2012}, this quantity $\tilde a_m(\Delta,l)$ lives in $!K(\RES)[[\mathbb A_{\ka}^1]^{-1}]_{\loc}$. Moreover, we are able to prove that $\tilde a_m$ does not depend on coordinates of $\Gamma^n$. Indeed, by definition of $\mu\Gamma^{\bdd}$, if $h$ is a morphism from $(\Delta_e,l|_{\Delta_e})$ to some $(\Delta',l')$, then
$$|h(\gamma)|+l'(h(\gamma))=|\gamma|+l(\gamma)=|\gamma|+e.$$ 
Now, rewriting (\ref{0eq0001}) as
\begin{align*}
\tilde a_m(\Delta,l)=\sum_{e\in\mathbb N}\sum_{\gamma\in \Delta_e(m)}[\mathbb A_{\ka}^1]^{-m|\gamma|-e}([\mathbb A_{\ka}^1]-1)^n
\end{align*}
and putting
$$\tilde a_{m,e}(\Delta,l)=\sum_{\gamma\in \Delta_e(m)}[\mathbb A_{\ka}^1]^{-m|\gamma|}([\mathbb A_{\ka}^1]-1)^n,$$ 
we have 
$$\tilde a_m(\Delta,l)=\sum_{e\in\mathbb N}\tilde a_{m,e}(\Delta_e,l)[\mathbb A_{\ka}^1]^{-e}.$$

Let $(X,f,1)$ be an object of the category $\mu_{\Gamma}\RES$ with $f(X)\subset V_{\gamma_1}\times\cdots\times V_{\gamma_n}$ (thus, $\val_{\rv}(f_i(x))=\gamma_i$ for every $x\in X$). As in \cite[Subsec. 8.2]{HL}, we define 
$$\tilde b_m(X,f,1)=[X]\left(\frac{[1]_1}{[\mathbb A_{\ka}^1]}\right)^{m|\gamma|}$$
if $m\gamma\in\mathbb Z^n$ and $\tilde b_m(X,f,1)=0$ otherwise. Similarly as in Subsection \ref{hm2012}, the tensor products $\tilde b_m\otimes \tilde a_{m,e}$ and $\tilde b_m\otimes \tilde a_m$ respectively induce morphisms $\tilde h_{m,e}$ and $\tilde h_m$,
$$\tilde h_{m,e}, \tilde h_m: K(\mu_{\Gamma}^G\RV[*])/I'_{\sp}\to !K(\RES)[[\mathbb A_{\ka}^1]^{-1}]_{\loc},$$
such that
$$\tilde h_m([(X,f,\alpha)]+I'_{\sp})=\sum_{e\in\mathbb N}\tilde h_{m,e}([(\alpha^{-1}(e),f)]+I'_{\sp})[\mathbb A_{\ka}^1]^{-e}.$$

Let $X$ be a bounded smooth rigid $\ka((t))^{\alg}$-variety endowed with a gauge form $\omega$. Then one can view $(X,\omega)$ as an object $(X,\alpha)$ of the category $\mu_{\Gamma}\VF^{\bdd}[*]$, where $\alpha=\val\circ\omega$. Here, by convention, an object $(X,f,\alpha)$ of $\mu_{\Gamma}\VF^{\bdd}[*]$ can be simply written as $(X,\alpha)$ if $f$ is clear. It is also a fact that the image of $[(X,\alpha)]$ under $\int$ belongs to $K(\mu_{\Gamma}^G\RV[*])/I'_{\sp}$. We use the morphisms $\tilde{\Theta}$, $\Upsilon$ and $\hl$ defined in Section \ref{prepare10}; even $\tilde{\Theta}\circ \tilde h_{m,e}\circ\int$ was already mentioned in the same section, which we now denote by $\hl_{m,e}$. Putting
\begin{align*}
\hl_m([(X,\alpha)]):=(\tilde{\Theta}\circ \tilde h_m)(\int[(X,\alpha)]),
\end{align*}
we obtain a morphism of rings $\hl_m: K(\GBSRig_{\ka((t))^{\alg}})\to \mathscr M_{\ka,\loc}^{\hat{\mu}}$.

\subsection{A comparison result}\label{star}
\begin{theorem}\label{comp}
Let $X$ be a bounded smooth rigid $k((t))^{\alg}$-variety with formal model affine of pure relative dimension $d$, and let $\omega$ be a gauge form on $X$. Assume in addition that, by viewing $X$ as an object $(X,\alpha)$ of $\mu_{\Gamma}\VF^{\bdd}[*]$ with $\alpha=\val\circ\omega$, $X$ satisfies the hypotheses of Proposition \ref{eq010}. Then, the following equalities
$$\loc\left(\Phi([X(m),\omega(m)])\right)=[\mathbb A_{\ka}^1]^{-d}\hl_m([X,\val\circ\omega])$$
and
$$\loc\left(\SE([X])\right)=[\mathbb A_{\ka}^1]^{-d}\hl([X])$$
hold in $\mathscr M_{\ka,\loc}^{\hat{\mu}}$.
\end{theorem}

\begin{proof}
By definition, we can view $\alpha^{-1}(e)$ as an object of $\vol\VF^{\bdd}[*]$, which also satisfies Proposition \ref{eq010}. By the construction of $\tilde h_m$, one has 
$$\hl_m([(X,\alpha)])=\sum_{e\in\mathbb N}\hl_{m,e}([(\alpha^{-1}(e))])[\mathbb A_{\ka}^1]^{-e}.$$
We deduce from \cite[Lem. 3.1.1]{HL} that, for each $e\in\mathbb N$, there exists a $\beta_e\in\Gamma$ such that $\alpha^{-1}(e)$ is $\beta_e$-invariant. Applying Proposition \ref{eq010} to $\hl_{m,e}$, $e\in\mathbb N$, 
\begin{align}\label{aaa}
\hl_{m,e}([\alpha^{-1}(e)])=[\widetilde{\alpha^{-1}(e)}[m]]=[\alpha^{-1}(e)[m,\beta']][\mathbb A_{\ka}^1]^{-md\beta'+d}
\end{align}
for any $\beta'\geq \beta_e$ in $\Gamma$. Let $\tilde{\mu}$ denote the motivic measure which takes values in $\mathscr M_{\ka}$. Then, the RHS of (\ref{aaa}) is nothing but $\loc\left([\mathbb A_{\ka}^1]^d\tilde{\mu}(\alpha^{-1}(e)(m))\right)$, and we have
\begin{align*}
\hl_m([(X,\alpha)])&=\loc\left([\mathbb A_{\ka}^1]^d\sum_{e\in\mathbb N}\tilde{\mu}(\alpha^{-1}(e)(m))[\mathbb A_{\ka}^1]^{-e}\right)\\
&=\loc\left([\mathbb A_{\ka}^1]^d\Phi([X(m),\omega(m)])\right)
\end{align*}
in the ring $\mathscr M_{\ka,\loc}$. Remark that, according to the proof of Lemma \ref{jussieu17may}, the previous identity also holds in $\mathscr M_{\ka,\loc}^{\hat{\mu}}$. Thus we are able to deduce the following
\begin{align*}
\sum_{m\geq 1}\hl_m([(X,\alpha)])T^m=[\mathbb A_{\ka}^1]^d\sum_{m\geq 1}\loc\left(\Phi([X(m),\omega(m)])\right)T^m
\end{align*}
in $\mathscr M_{\ka,\loc}^{\hat{\mu}}[[T]]$. By this, the rationality of the series $\sum_{m\geq 1}\hl_m([(X,\alpha)])T^m$ in $\mathscr M_{\ka,\loc}^{\hat{\mu}}[[T]]$ follows from that of the Poincar\'e power series, and the limit $\lim_{T\to\infty}$ lives in $\mathscr M_{\ka,\loc}^{\hat{\mu}}$. Since 
$$-\SE([X])=\lim_{T\to\infty}\sum_{m\geq 1}\Phi([X(m),\omega(m)])T^m$$
does not depend on $\omega$, so does $\lim_{T\to\infty}\sum_{m\geq 1}\hl_m([(X,\alpha)])T^m$. Hence, we can choose a gauge form $\omega$ on $X$ so that $\alpha=0$, and by Proposition \ref{eq010},  
$$-\hl([X])=\lim_{T\to\infty}\sum_{m\geq 1}\hl_m([(X,\alpha)])T^m$$
as desired.
\end{proof}

Let $X^{\star}$ be a line bundle over a bounded smooth rigid $\ka((t))^{\alg}$-variety $X$. The below corollary will be used in Subsection \ref{x1} for the following rigid $\ka((t))^{\alg}$-varieties
\begin{align}\label{xxx1}
X_1=\left\{(x,y,z)\in\VF^d\
\begin{array}{|l}
\val(x)\geq 0, \val(y)>0, \val(z)>0\\
x\not=0 \ \text{and}\ y\not=0, f(x,y,z)=t
\end{array}
\right\}
\end{align}
and
\begin{align*}
X_1^{\star}=\left\{(x,y,z)\in\VF^d\ 
\begin{array}{|l}
\val(x)\geq 0, \val(y)>0, \val(z)>0\\
x\not=0 \ \text{and}\ y\not=0, \rv(f(x,y,z))=\rv(t)
\end{array}
\right\},
\end{align*}
where $f$ is the formal series given in Theorem \ref{full}.

\begin{corollary}\label{comp11}
With $X$ as in Theorem \ref{comp}, 
$$\loc\left(\SE([X])\right)=[\mathbb A_{\ka}^1]^{-d}\hl([X^{\star}])$$ 
in $\mathscr M_{\ka}^{\hat{\mu}}$.
\end{corollary}

\begin{proof}
Note that $[X^{\star}]=[X][\mathfrak m]$; so, the statement follows from Theorem \ref{comp} and from the fact that $\SE([\mathfrak m])=1$ (see the formula (\ref{openball}) of Subsection \ref{x0}).
\end{proof}


\section{Proof of the formal version (Theorem \ref{full})}\label{sec7}
In this section, $K$ is $\ka((t))$ and $R$ is $\ka[[t]]$ with $\ka$ a field of characteristic zero.

\subsection{Remark}
It is worthy to notice that, thanks to Theorem \ref{comp} and to Corollary \ref{comp11}, which express the link between the morphisms $\SE$ and $\hl$, we are able to deduce Theorem \ref{full} from the proof of Theorem \ref{conj} given in Section \ref{proofthm1}. 

However, we shall give an alternative computation on the morphism $\SE$ in Subsections \ref{s8.2} and \ref{x0}. The reason is that below arguments are also valid  for every valued field of equal characteristic zero, not necessarily algebraically closed, and that the setting is in $\mathscr M_{\ka}^{\hat{\mu}}$. Thus, in the case where we could compute $\SE([X_1])$ with $X_1$ in (\ref{xxx1}), without the condition that $\ka$ is an algebraically closed field, then Conjecture \ref{conj1} would be answered in the most general version.

\subsection{The left hand side}\label{s8.2}
As in Section \ref{introduction}, let $f(x,y,z)$ be a formal power series in $\ka[[x,y,z]]$, where $(x,y,z)$ is a system of coordinates of $\ka^d=\ka^{d_1}\times\ka^{d_2}\times\ka^{d_3}$, satisfying the hypotheses 
\begin{align}\label{eq000}
\begin{cases}
f(\tau x,\tau^{-1}y,z)=f(x,y,z),\ \tau\in \ka^{\times}\\
f(0,0,0)=0.
\end{cases}
\end{align}
Let $\mathfrak X$ be the formal completion of $\mathbb A_{\ka}^d$ along $\mathbb A_{\ka}^{d_1}$ with structural morphism $f_{\mathfrak X}$ induced by $f$ (see Conjecture \ref{conj1}). Then, $\mathfrak X$ is a generically smooth special formal $R$-scheme of relative dimension $d-1$, and $\mathfrak X_0=\mathbb A_{\ka}^{d_1}$. Let $\mathfrak X_{\eta}$ be the generic fiber of $\mathfrak X$. By Corollary \ref{sosanh} we have the following identity
$$\SE([\mathfrak X_{\eta}])=[\mathbb A_{\ka}^1]^{-d+1}\int_{\mathbb A_{\ka}^{d_1}}\mathcal S_{f_{\mathfrak X}},$$
and hence
\begin{align}\label{equation1}
\int_{\mathbb A_{\ka}^{d_1}}\mathcal S_{f_{\mathfrak X}}=[\mathbb A_{\ka}^1]^{d-1}\SE([\mathfrak X_{\eta}]),
\end{align}
both live in the ring $\mathscr M_{\ka}^{\hat{\mu}}$. Note that $\mathfrak X_{\eta}$ is of the form
\begin{align*}
\mathfrak X_{\eta}=\left\{(x,y,z)\in\mathbb A_{\widehat{K^s},\Rig}^d\
\begin{array}{|l}
\val(x)\geq 0, \val(y)>0, \val(z)>0\\
f(x,y,z)=t
\end{array}
\right\},
\end{align*}
where $\val(x)$ stands for $\min_i\{\val(x_i)\}$. Now, write $\mathfrak X_{\eta}$ as disjoint union of definable subsets 
$$X_0=\left\{(x,y,z)\in \mathfrak X_{\eta} \mid x=0\ \text{or}\ y=0\right\}$$ 
and 
$$X_1=\left\{(x,y,z)\in \mathfrak X_{\eta} \mid x\not=0\ \text{and}\ y\not=0\right\}.$$ 


\subsection{The right hand side}\label{x0}
By the homogeneity (\ref{eq000}), whenever $x=0$ or $y=0$, $f(x,y,z)=f(0,0,z)$. Thus, we can decompose $X_0$ into cartesian product $Y_0\times Z_0$, where
$$Y_0=\left\{(x,y)\in \mathbb A_{\widehat{K^s},\Rig}^{d_1+d_2} \mid \val(x)\geq 0, \val(y)>0, x=0\ \text{or}\ y=0\right\}$$
and
$$Z_0=\left\{z\in\mathbb A_{\widehat{K^s},\Rig}^{d_3} \mid \val(z)>0, f(0,0,z)=t\right\}.$$
Moreover, $Y_0$ is a disjoint union of subsets 
$$Y_{0,1}=\left\{x\in \mathbb A_{\widehat{K^s},\Rig}^{d_1} \mid 0\leq \val(x)<\infty\right\}$$
and 
$$Y_{0,2}=\left\{y\in \mathbb A_{\widehat{K^s},\Rig}^{d_2} \mid \val(y)>0\right\},$$ 
therefore
$$X_0=\Big(Y_{0,1}\times Z_0\Big)\sqcup \Big(Y_{0,2}\times Z_0\Big).$$

Now, by Proposition \ref{cor3}, 
\begin{align}\label{eq203}
\SE([Y_{0,1}\times Z_0])=-\lim_{T\to \infty}\sum_{m\geq 1}\Phi\Big([Y_{0,1}(m)\times Z_0(m),dx\times\omega(m)]\Big)T^m,
\end{align}
where $\omega$ is an appropriate gauge form on $Z_0$ and $dx$ denotes $dx_1\wedge\dots\wedge dx_{d_1}$. By Proposition \ref{gsr}, 
$$\Phi\Big([Y_{0,1}(m)\times Z_0(m),dx\times\omega(m)]\Big)=\Phi\Big([Y_{0,1}(m),dx(m)]\cdot\Phi\Big([Z_0(m),\omega(m)]\Big).$$
Note that $\Phi\Big([Y_{0,1}(m),dx(m)]=0$, because
$$Y_{0,1}(m)=\left\{x\in \mathbb A_{K(m),\Rig}^{d_1} \mid \val(x)\geq 0\right\}\setminus\{0\},$$ 
and $\Phi\Big([\{x\in \mathbb A_{K(m),\Rig}^{d_1} \mid \val(x)\geq 0\},dx(m)]\Big)=1$ and $\Phi\Big([\{0\},1]\Big)=1$. It thus follows from (\ref{eq203}) that 
\begin{equation}\label{eq204}
\SE([Y_{0,1}\times Z_0])=0.
\end{equation}

Again, by Proposition \ref{cor3}, we have 
$$
\SE([Y_{0,2}\times Z_0])=-\lim_{T\to \infty}\sum_{m\geq 1}\Phi\Big([Y_{0,2}(m)\times Z_0(m),d'x\times\omega(m)]\Big)T^m,$$
where $\omega$ is an appropriate gauge form on $Z_0$ and $d'x$ stands for $dx_1\wedge\dots\wedge dx_{d_2}$. Since $\Phi$ is a morphism of rings due to Proposition \ref{gsr}, 
$$\Phi\Big([Y_{0,2}(m)\times Z_0(m),d'x\times\omega(m)]\Big)=\Phi\Big([Y_{0,2}(m),d'x(m)]\cdot\Phi\Big([Z_0(m),\omega(m)]\Big).$$
A simple computation gives 
\begin{align}\label{openball}
\Phi\Big([Y_{0,2}(m),d'x(m)]\Big)=[\mathbb A_{\ka}^1]^{-d_2}
\end{align}
for any $m\geq 1$. This follows that
$$\SE([Y_{0,2}\times Z_0])=[\mathbb A_{\ka}^1]^{-d_2}\SE([Z_0]);$$
and, moreover, this identity also holds in $\mathscr M_{\ka}^{\hat{\mu}}$, by Corollary \ref{sosanh}.

Now, let us show how $Z_0$ concerns the RHS of the integral identity. As denoted in Conjecture \ref{conj1}, $\mathfrak Z$ is the formal completion of $\mathbb A_{\ka}^{d_3}$ at the origin with structural morphism $f_{\mathfrak Z}$ induced by $f(0,0,z)$; hence it has the relative dimension $d_3-1$. Furthermore, $Z_0$ is nothing but the generic fiber $\mathfrak Z_{\eta}$. By this and by Corollary \ref{sosanh} we have
$$\SE([Z_0])=[\mathbb A_{\ka}^1]^{-d_3+1}\mathcal S_{f_{\mathfrak Z}},$$
thus 
\begin{equation}\label{eq205}
\SE([Y_{0,2}\times Z_0])=[\mathbb A_{\ka}^1]^{-d_2-d_3+1}\mathcal S_{f_{\mathfrak Z}},
\end{equation}
which hold in  $\mathscr M_{\ka}^{\hat{\mu}}$. Therefore, we deduce from (\ref{eq204}) and (\ref{eq205}) that
\begin{align}\label{eq206}
[\mathbb A_{\ka}^1]^{d_1}\mathcal S_{f_{\mathfrak Z}}=[\mathbb A_{\ka}^1]^{d-1}\SE([X_0])
\end{align}
in $\mathscr M_{\ka}^{\hat{\mu}}$.


\subsection{Conclusion}\label{x1}
At the moment, the field $\ka$ is assumed to be algebraically closed. We consider the theory $\ACVF(0,0)$ with base structure $\ka((t))$, this theory is valid for rigid $\widehat{K^s}$-varieties. After (\ref{equation1}) and (\ref{eq206}), we want to prove, in the present context, that $\loc\left(\SE([X_1])\right)=0$ in $\mathscr M_{\ka,\loc}^{\hat{\mu}}$; and thanks to Corollary \ref{comp11}, it suffices to verify that $\hl([X_1^{\star}])=0$. The latter was already realized in Subsection \ref{bbb}. Theorem \ref{full} has completely proved.

\begin{remark}
By the above computation, proving Conjecture \ref{conj} in its full version is equivalent to showing $\SE([X_1])=0$ in $\mathscr M_{\ka}^{\hat{\mu}}$.
\end{remark}

\begin{ack}
I am grateful to my advisor F. Loeser for various supports and encouragements. Mathematically, his work joint with E. Hrushovski on the Lefschetz fixed point formula \cite{HL} is crucial for proofs in the present article, and the proof of Lemma \ref{cachan6} is under his idea. I would like to thank J. Nicaise, J. Sebag and N. Budur for many valuable comments to this work. Finally, the contributions of the referees to the article are important and I sincerely thank them for those.
 
This work is partially supported by the European Research Council under the European Community's Seventh Framework Programme (FP7/2007-2013) / ERC Grant Agreement $\text{n}^{\circ}$ 246903/NMNAG.
\end{ack}


\end{document}